\documentclass[reqno, a4paper,12pt]{amsart}

\usepackage[capposition=bottom]{floatrow}
\usepackage{euscript,eufrak,verbatim}
\usepackage{graphicx}
\usepackage[usenames]{color}
\usepackage[colorlinks,linkcolor=red,anchorcolor=blue,citecolor=blue]{hyperref}
\usepackage{amsmath, mathtools}
\usepackage{mathtools}
\usepackage{amsthm}
\usepackage[all]{xy}
\usepackage{amssymb} 
\usepackage{bm}
\usepackage{booktabs,siunitx}

\usepackage[all]{xy}

\usepackage{mathrsfs}
\usepackage{amscd}

\usepackage{enumitem}
\usepackage[numbers]{natbib}

\usepackage{mathptmx}
\usepackage[T1]{fontenc}

\usepackage{lipsum}
\usepackage{accents}
\usepackage{titlesec}
\usepackage{calrsfs}

\makeatletter
\numberwithin{equation}{section}

\theoremstyle{cuplain}
\newtheorem{main theorem}{Main Theorem}
\newtheorem{theorem}{Theorem}[section]
\newtheorem{lemma}[theorem]{Lemma}

\newtheorem{proposition}[theorem]{Proposition}
\newtheorem{claim}[theorem]{Claim}

\newtheorem*{theorem*}{``Theorem''}
\theoremstyle{definition}
\newtheorem{definition}[theorem]{Definition}
\newtheorem{remark}[theorem]{Remark}
\newtheorem{example}[theorem]{Example}

\newtheorem{question}[theorem]{Question}
\newtheorem*{example*}{Example}
\newtheorem*{remark*}{Remark}

\newtheoremstyle{break}
{\topsep}{\topsep}%
{\itshape}{}%
{\bfseries}{}%
{\newline}{}%
\theoremstyle{break}

\newtheoremstyle{break}
{\topsep}{\topsep}%
{\normalshape}{}%
{\bfseries}{}%
{\newline}{}%

\numberwithin{equation}{section}
\newcommand{\vep}{\varepsilon}

\newcommand{\norm}[1]{\left\lVert#1\right\rVert}
\newcommand{\setcond}{\hspace{2pt} \middle| \hspace{2pt}}

\DeclareFontFamily{U}{stix2bb}{\skewchar\font127 }
\DeclareFontShape{U}{stix2bb}{m}{n} {<-> stix2-mathbb}{}
\DeclareMathAlphabet{\mathbb}{U}{stix2bb}{m}{n}
\DeclareMathAlphabet{\pazocal}{OMS}{zplm}{m}{n}

\newcommand\lbar[1]{%
\underaccent{\bar}{#1}}

\newcommand{\colA}[1]{\makebox[4.5em][c]{$#1$}}
\newcommand{\colB}[1]{\makebox[4.5em][c]{$#1$}}

\newcommand{\RN}[1]{%
\textup{\uppercase\expandafter{\romannumeral#1}}%
}

\titleformat{\section}
{\normalfont\bfseries\Large} 
{\thesection} 
{1em} 
{\bfseries} 
\titleformat{\subsection}
{\normalfont\bfseries\normalsize}
{\thesubsection} 
{1em} 
{\bfseries}

\begin{document}


\newcommand\titlelowercase[1]{\texorpdfstring{\lowercase{#1}}{#1}}

\font\mathptmx=cmr12 at 12pt



\title[\fontsize{13}{12}\mathptmx {\it{L\titlelowercase{yapunov exponents for random products of non-negative matrices}}}]{\LARGE L\titlelowercase{yapunov exponents for random products of} \protect{\\[5pt]}\titlelowercase{non-negative matrices}}


\author[\fontsize{13}{12}\mathptmx {\it{N\titlelowercase{ima} A\titlelowercase{libabaei}}}]{\fontsize{13}{12}\mathptmx N\titlelowercase{ima} A\titlelowercase{libabaei}}



\maketitle

\begin{abstract}
We first consider i.i.d.\ products of finitely many invertible $2 \times 2$ matrices with positive entries, and prove that the top Lyapunov exponent admits an explicit, rapidly convergent Neumann-series-type representation involving an infinite matrix. This yields a polynomial-time algorithm for computing the Lyapunov exponent. We further show that non-negative invertible $2 \times 2$ matrices are simultaneously conjugate to positive matrices if and only if ``generalized'' heteroclinic connections do not occur among products of length at most $2$.

These results yield a series formula for the Hausdorff dimension of the intersection of the middle-$n$th Cantor set with a random translate of itself, for every natural number $n$ except $4$. Furthermore, our method applies to the intersection of thick Cantor sets under random translation. We also determine the almost sure growth rate of i.i.d.\ three-term recurrences with finitely many positive coefficients.
\end{abstract}

\section{Introduction} \label{section: introduction}

Let $\{A_i\}_{i\in I}$ be a finite family of invertible non-negative $2\times2$ matrices, and let $(w_i)_{i\in I}$ be positive weights with $\sum_{i \in I} w_i = 1$. We consider i.i.d.\ products $A_{i_n} \cdots A_{i_1}$ driven by the Bernoulli measure $\mu$ on $I^{\mathbb N}$ with weights $(w_i)_{i\in I}$. Then, by Furstenberg--Kesten \cite{Furstenberg--Kesten} the following limit exists and is constant for $\mu$-a.e.\ sequence $(i_n)_{n \in \mathbb N}$,
\[
\lambda \;=\; \lim_{n\to\infty}\frac1n \log \|A_{i_n}\cdots A_{i_1}\|,
\]
independent of the choice of norm. This $\lambda$ is called the \textbf{(top) Lyapunov exponent}. This quantity naturally arises, for example, in the context of dimension theory for the intersection of a Cantor-type fractal with a random translate of itself.

Kingman emphasized the calculation of Lyapunov exponents as one of the central unsolved problems in subadditive ergodic theory \cite[p.~897]{Kingman}. Although the existence of Lyapunov exponents is classical, obtaining an explicit, fast-converging expression for $\lambda$ remained challenging even for $2 \times 2$ matrices.

A breakthrough was achieved by Pollicott \cite{Pollicott10}, who represented the Lyapunov exponent of the random products of positive matrices through the root of the Fredholm determinant of a suitable transfer operator. The Fredholm determinant is an entire function whose $n$-th coefficient can be determined by computing all products of matrices of length $n$, resulting in an algorithm for approximating $\lambda$. However, this method eventually suffers from a combinatorial growth in the number of products that must be computed. For $2\times2$ matrices, the resulting computational cost is subexponential: there is $\alpha > 0$ such that to obtain truncation error at most $\vep$, one needs
\[
O \! \left( \exp{ \left( \alpha \sqrt{ \log{ \left( 1/ \vep \right) } } \right) } \right)
\]
arithmetic operations. Later, Jurga--Morris \cite{Jurga--Morris} and Pollicott \cite{Pollicott20} refined this approach and obtained better approximations and bounds, with the same arithmetic complexity.

In this paper, we develop a new approach, \textbf{Kernel expansion}, which yields the following convergent expression for the Lyapunov exponent for (entrywise) \emph{positive} $2\times 2$ matrices (Theorem \ref{thm: main theorem 1}):
\begin{equation} \label{eq: explicit representation in the introduction}
\lambda \;=\; \sum_{n=0}^{\infty}\bigl(T^{n}\boldsymbol{v}\bigr)_0.
\end{equation}
Here, $T$ is a suitable infinite matrix acting on a subspace of bounded sequences, $\boldsymbol{v} = (v_n)_{n=0}^\infty$ is a bounded vector, and $(\cdot)_0$ denotes the $0$-th coordinate. Moreover, truncating the infinite matrix $T$ results in an algorithm with \emph{polynomial} complexity; to achieve error $\vep>0$, we only need
\[
O\!\left( \left(\log(1/\vep) \right)^3 \right)
\]
arithmetic operations (Theorem \ref{thm: main theorem 2}).

We then identify a simple dynamical condition under which we can apply our machinery to \emph{non-negative} matrices. We say that a \textbf{heteroclinic connection} exists when some matrix maps an unstable direction of one matrix onto a stable direction of another. We show that non-negative $2 \times 2$ matrices $\{A_i\}_{i \in I}$ are simultaneously conjugate to positive matrices by a linear change of basis if and only if the finite set $\{A_i\}_{i\in I}\cup\{A_iA_j\}_{i,j\in I}$ has no heteroclinic connections in a ``generalized'' sense (see Definition \ref{def: generalized depth2 heteroclinic connection} and Theorem \ref{thm: main theorem 3}). In particular, under this condition we obtain the same series representation for $\lambda$.

\medskip

As an application, we can determine the Hausdorff dimension of the intersections of translated Cantor sets. The classical result of Hawkes \cite[Theorem 1]{Hawkes} states that for the middle-third Cantor set $\pazocal{C}_3$, one has
\begin{equation*}
\mathrm{dim}_{\mathrm{H}} \Big( \pazocal{C}_3 \cap (\pazocal{C}_3 + t) \Big) = \frac{1}{3} \frac{\log{2}}{\log{3}} = 0.2103099178571524\cdots
\end{equation*}
for Lebesgue-a.e.\ $t \in \mathbb{T}$. Here, a closed formula is available due to a special arithmetic structure, the \emph{difference-set condition}. For general Cantor sets, Kenyon--Peres \cite{Kenyon--Peres: translated Cantor sets} showed that the dimension is governed by the Lyapunov exponent of random products of non-negative matrices (Theorem \ref{theorem: KP main result}). 

Using our method, we can calculate the Hausdorff dimension of $\pazocal{C}_n \cap (\pazocal{C}_n + t)$ for every $n\geq 5$, where $\pazocal{C}_n\subset \mathbb{T} := \mathbb{R}/\mathbb{Z}$ is the middle-$n$th Cantor set (Proposition \ref{prop: toothless twins}). For instance,
\begin{align*}
& \mathrm{dim}_{\mathrm{H}} \Big( \pazocal{C}_5 \cap (\pazocal{C}_5 + t) \Big) \\
&
\hspace{10pt}=
0.7203495993043838851551910620020220129260634798485094974824780505 \\
&\hspace{25pt}\phantom{0.} 3336327277508886023310650356822860709630313496412358829583177545\cdots.
\end{align*}
This gives the first $128$ certified digits for $n=5$ (middle-fifth Cantor set) and Lebesgue-a.e.\ $t$. We remark that difference-set condition fails for every $n$ except $3$.

\bigskip

The paper is organized as follows. In this section, we introduce the problem and state the main results. Then we present two applications: to intersections of translated Cantor sets and to random three-term recurrence relations.

Section \ref{section: definitions and basic properties} begins with a heuristic explanation of our novel method, \emph{Kernel expansion}: we consider the variational integral formula by Furstenberg--Kifer, and expand the integrand with functions that integrate to $0$ under stationary probability measures. This yields a Neumann-series type representation using an infinite matrix, although the na\"ive approach using operator contraction of $T$ seems to fail. We prove equation \eqref{eq: explicit representation in the introduction} in Section \ref{section: proof of main theorem 1} through a dynamical interpretation of the series and contraction argument.

Section \ref{section: proof of main theorem 2} proves Theorem \ref{thm: main theorem 2}, which assures the polynomial-time approximation. This involves a delicate evaluation of an oscillating combinatorial sum, whose absolute convergence is not guaranteed. Our approach is to derive integral representations for the full operator $T$ and for its finite approximation $T_m$, and then to analyze the difference of the corresponding integral operators.

Section \ref{section: proof of main theorem 3 and applications} proves the equivalence of absence of generalized heteroclinic connections and simultaneous conjugacy to positive matrices (Theorem \ref{thm: main theorem 3}), and then proves the applications. In Section \ref{section: generalizations and limitations}, we briefly discuss the higher-dimensional analogue and its limitations. Finally, Appendix \ref{appendix: comparison} records a numerical comparison to the transfer-operator approaches by Jurga--Morris \cite{Jurga--Morris} (who refined the method by Pollicott \cite{Pollicott10}).

\subsection{Preliminaries} \label{subsection: preliminaries}

Furstenberg and Kifer \cite{Furstenberg--Kifer} proved that the top Lyapunov exponent $\lambda$ can be expressed as an integral with respect to stationary measures on the real projective space.

We denote by $\mathrm{GL}_2(\mathbb{R})$ the set of invertible $2 \times 2$ matrices with real entries. We also write
\begin{equation} \label{eq: definition of PGL2}
\mathrm{PGL}_2(\mathbb{R}) = \mathrm{GL}_2(\mathbb{R}) \big/ \sim,
\end{equation}
where $A \sim B$ if and only if there is $t \in \mathbb{R} \setminus \{0\}$ with $A = tB$. For a matrix $M \in \mathrm{GL}_2(\mathbb{R})$, we denote its equivalence class by $[M] \in \mathrm{PGL}_2(\mathbb{R})$.

Denote the Riemann sphere by $\widehat{\mathbb{C}} = \mathbb{C} \cup \{\widehat{\infty}\}$. We identify an element $[M] = \left[ \begin{pmatrix}
\alpha & \beta \\
\gamma & \delta
\end{pmatrix} \right] \in \mathrm{PGL}_2(\mathbb{R})$ with the M\"obius transformation $f: \widehat{\mathbb{C}} \to \widehat{\mathbb{C}}$ defined by, with the usual convention at poles and at $\widehat{\infty}$,
\[
f(x) = \frac{\alpha x + \beta}{\gamma x + \delta} \qquad \left( x \in \widehat{\mathbb{C}} \right).
\]
By abuse of notation we write $f = [M]$. Also, let $\norm{\cdot}_1$ denote the $\ell^1$-norm, the sum of the absolute values of components of a vector.

Now, consider the $1$-simplex
\[
\Delta = \left\{ v = \begin{pmatrix} v_1 \\ v_2 \end{pmatrix} \in \mathbb{R}^2 \setcond v_1, v_2 \geq 0, v_1 + v_2 = 1 \right\}.
\]
Then, an invertible non-negative $2 \times 2$ matrix $A$ induces a map $\Delta \to \Delta$ by
\[
v \mapsto \frac{A\,v \;}{ \norm{A\,v}_1} \qquad \left( v \in \Delta \right).
\]
We identify $\Delta$ and $[-1, 1]$ via the map
\[
\psi:[-1, 1] \to \Delta, \qquad \psi(x) = \frac{1}{2} \begin{pmatrix} 1+x \\ 1-x \end{pmatrix}.
\]
Under this identification, the projective action above corresponds to a M\"obius transformation: define $F: \mathrm{GL}_2(\mathbb{R}) \to \mathrm{GL}_2(\mathbb{R})$ by
\begin{equation} \label{eq: definition of F}
F\left(
\begin{pmatrix}
p & q \\
r & s
\end{pmatrix}
\right) =
\frac{1}{2}
\begin{pmatrix}
p-q-r+s & p+q-r-s \\
p-q+r-s & p+q+r+s
\end{pmatrix}.
\end{equation}
Then, the action of $A$ on $[-1,1]$ is the M\"obius transformation $[F(A)] \in \mathrm{PGL}_2(\mathbb{R})$.

\smallskip
Let $\{A_i\}_{i \in I}$ be a finite family of non-negative invertible matrices and let $f_i = [F(A_i)]$ be their projective M\"obius maps for each $i \in I$. Let $(w_i)_{i\in I}$ be a probability vector, i.e.\ $\sum_{i \in I} w_i = 1$ and $w_i > 0$ for all $i \in I$. We consider the Bernoulli measure $\mu$ on $I^{\mathbb N}$ with weights $(w_i)_{i\in I}$.

A probability measure $\nu$ on $[-1,1]$ is said to be \textbf{$\boldsymbol{\mu}$-stationary for} $\boldsymbol{\{f_i\}_{i \in I}}$ if for every continuous function $g$,
\[ \int_{[-1, 1]} g \, d\nu = \sum_{i \in I} w_i \int_{[-1, 1]} \big( g \circ f_i \big) (x) \, d\nu(x). \]
Then, Furstenberg--Kifer proved the following formula. \cite[Theorem 2.2]{Furstenberg--Kifer}
\begin{align} \label{eq: integral form by Furstenberg Kifer}
\lambda
= \sup \left\{ \,
\sum_{i \in I} w_i
\int_{[-1, 1]} 
\log \norm{
A_i \psi(x)
}_1 \, d\nu(x)
\setcond \text{$\nu$ is $\mu$-stationary} \right\}.
\end{align}

\subsection{Main results} \label{subsection: Main results}

Let $\ell^\infty(\mathbb{N}_0)$ be the set of bounded sequences of complex numbers, indexed by $\mathbb{N}_0 = \mathbb{N} \cup \{0\}$. For $x, c \in \mathbb{C}$ and $|x| \leq 1$, we define $v(x \, ; \, c) \in \ell^\infty(\mathbb{N}_0)$ by
\begin{equation*}
\left( v(x \, ; \, c) \right)_n =
\begin{dcases}
c & (n = 0), \\
- \frac{(-x)^n}{n} & (n \geq 1)
\end{dcases}.
\end{equation*}
Let $V \subset \ell^\infty(\mathbb{N}_0)$ be the algebraic linear span of vectors of the form $v(x \, ; \, c)$.
\[
V = \mathrm{Span}_{\mathbb{C}} \left\{ \, v(x\,;\,c) \setcond x, c \in \mathbb{C},\; |x| \leq 1 \, \right\}.
\]

For an infinite matrix $T=(b_{k,n})_{k,n\in\mathbb N_0}$ and $u = (u_n)_{n = 0}^\infty \in \ell^\infty(\mathbb{N}_0)$, we define $Tu$ by
\[
\big( Tu \big)_k = \sum_{n = 0}^\infty b_{k,n} u_n \quad \big( k \in \mathbb N_0 \big)
\]
whenever the series converges for all $k \in \mathbb N_0$.

We obtain the following explicit expression for the Lyapunov exponent of invertible positive $2 \times 2$ matrices. For a M\"obius map $f = [M] \in \mathrm{PGL}_2(\mathbb{R})$, define its ``transpose'' by the M\"obius map represented by $M^\top$, that is, $f^{\top} = [M^{\top}]$ (which does not depend on the choice of $M$). Also, write $f'$ to denote the usual differentiation of $f$. Note that for every invertible positive matrix $A$ we have $\frac{\det F(A)}{f_A'(0)} > 0$, where $f_A = [F(A)]$.

\begin{theorem}\label{thm: main theorem 1}
Let $\{A_i\}_{i\in I}$ be a finite family of invertible positive $2 \times 2$ matrices, and let $\mu$ be the Bernoulli measure on $I^{\mathbb N}$ associated to a probability vector $(w_i)_{i\in I}$. Then the top Lyapunov exponent $\lambda$ admits the explicit representation
\begin{equation} \label{eq: lambda representation in main theorem}
\lambda=\sum_{n=0}^\infty \bigl(T^n \boldsymbol{v}\bigr)_0.
\end{equation}
Here, letting $f_i = [F(A_i)]$,
\[
\boldsymbol{v} = \sum_{i \in I} w_i \; v \left( f_i(0) \, ; \, \frac{1}{2} \log \left( \frac{\det F(A_i)}{f_i'(0)} \right) \right),
\]
and the linear operator $T=(b_{k,n})_{k,n\in\mathbb N_0}: V \to V$ is defined by
\begin{align*}
& b_{k,0}=0 \quad \text{for } k \geq 0, \hspace{45pt} b_{0,n}=\sum_{i\in I} w_i\, \, {\big( f_i^{\top}(0) \big)}^n \quad \text{for } n\geq1,\\
& b_{k,n}=\sum_{i\in I} w_i \sum_{\ell=1}^{\min\{k,n\}}
\binom{n}{\ell}\binom{k-1}{\ell-1} \, \, 
{\big( f_i^{\top}(0) \big)}^{n-\ell}\big(-f_i(0)\big)^{k-\ell} {\big( f_i'(0) \big)}^{\ell}\quad \text{for } k,n\geq1.
\end{align*}
\end{theorem}

We remark that $T(V) \subset V$ is not apparent, and we will prove this. The coefficients of $T$ are obtained by a technique we call \emph{Kernel expansion}: we expand the integrand in \eqref{eq: integral form by Furstenberg Kifer} with functions $\{1\} \cup \{\phi_n\}_{n \in \mathbb N}$ on $[-1,1]$, where $\phi_n$ are ``Kernel functions'' in the sense
\[
\int_{[-1,1]} \phi_n \, d\widetilde\nu = 0
\]
for every $\mu$-stationary measure $\widetilde \nu$ for $\{f_i^\top\}_{i\in I}$. Subsection \ref{subsection: Heuristic explanation} explains this in detail.

The following theorem shows that the finite approximation converges exponentially fast, resulting in an algorithm whose complexity is polynomial in $\log\big(1/\vep)$ to achieve error $\leq \vep$. For the precise bounds, see Theorem \ref{thm: precise bounds}.
\begin{theorem} \label{thm: main theorem 2}
Under the assumptions of Theorem \ref{thm: main theorem 1}, for any $\vep > 0$ there are natural numbers $N$ and $M$ such that $N = O \! \left( \log{(1/\vep)} \right)$ and $M = O \! \left( \log{(N/\vep)} \right)$ with
\[
\left| \lambda - \sum_{n=0}^{N-1} \big( {T_M}^n \boldsymbol{v}^{(M)} \big)_0 \right| < \vep,
\]
where $T_M$ is the upper-left $M \times M$ submatrix of $T$, and $\boldsymbol{v}^{(M)}$ is the first $M$ components of $\boldsymbol{v}$. In particular, $\lambda$ can be approximated with $O \! \left( \left( \log{ \left( 1/\vep \right) } \right)^3 \right)$ arithmetic operations.
\end{theorem}

The series representation extends to non-negative families, provided the associated projective maps admit no generalized heteroclinic connections (Definition \ref{def: generalized depth2 heteroclinic connection}). We show in Theorem \ref{thm: main theorem 3} that the absence of generalized heteroclinic connections is equivalent to the existence of an invariant arc under the projective actions.

\begin{definition}
Let $\widehat{\mathbb{R}} = \mathbb{R} \cup \{\widehat{\infty}\} \subset \widehat{\mathbb{C}}$. A subset of $\widehat{\mathbb{R}}$ homeomorphic to $[0,1]$ is called a closed arc. A finite family of M\"obius transformations $\{f_i\}_{i \in I}$ is said to admit an \textbf{invariant arc} if there is a closed arc $J$ such that $f_i(J) \subset \mathring{J}$ for all $i \in I$, where $\mathring{J}$ is the relative interior of $J$ in $\widehat{\mathbb{R}}$.
\end{definition}

\begin{theorem} \label{thm: main theorem 3}
Let $\{A_i\}_{i\in I}$ be a finite family of invertible non-negative $2 \times 2$ matrices, and let $f_i = [F(A_i)]$ be the projective M\"obius map for each $i \in I$. Then, the following are equivalent.
\begin{enumerate}
\item $\{A_i\}_{i \in I}$ has no generalized heteroclinic connections of depth $2$.
\item $\{f_i\}_{i \in I}$ admits an invariant arc.
\item There is $P \in \mathrm{GL}_2(\mathbb{R})$ such that $M_i = P A_i P^{-1}$ is a positive matrix for every $i \in I$.
\end{enumerate}
\end{theorem}

Assume (3) and take $P$ so that $M_i = P A_i P^{-1}$ is a positive matrix for all $i \in I$. Notice that for any $ (i_1, \ldots, i_n) \in I^n$,
\[
M_{i_n} \cdots M_{i_1} = P A_{i_n} \cdots A_{i_1} P^{-1}.
\]
Therefore, the Lyapunov exponent of $\{A_i\}_{i \in I}$ and $\{M_i\}_{i \in I}$ coincide, and we can apply Theorem \ref{thm: main theorem 1} to $\{M_i\}_{i \in I}$, thereby obtaining the series representation for the Lyapunov exponent of $\{A_i\}_{i \in I}$.

\begin{remark}
For a non-negative matrix family with generalized heteroclinic connections of depth $2$, the formally defined partial sum in equation \eqref{eq: lambda representation in main theorem} may converge to the true value, or may oscillate and fail to converge. See Example \ref{example: matrices with generalized heteroclinic connections}.
\end{remark}

\subsection{Application \RN{1}: Intersection of translated Cantor sets} \label{subsection: application I: Intersection of translated Cantor sets}

As an application, we determine the Hausdorff dimension of the intersection of randomly translated Cantor sets for several interesting classes. Inspired by Hawkes \cite{Hawkes}, Kenyon and Peres later tackled a general class of $b$-adic Cantor sets defined by digit sets: let $b \geq 2$ be an integer, and for $D_j \subset \{0, 1, \ldots, b-1\}$ ($j = 1, 2$), define
\begin{equation*}
K_j = \left\{ \sum_{n = 1}^\infty \frac{d_n}{b^n} \setcond d_n \in D_j \right\} \subset \mathbb{T} = \mathbb{R}/\mathbb{Z}.
\end{equation*}
Then, the object of interest is $\mathrm{dim}_{\mathrm{H}} \left( (K_1 + t) \cap K_2 \right)$. Hawkes' method (which yields a closed formula) does not apply in full generality. Roughly speaking, it uses the fact that $D_2 - D_1$ is contained in some arithmetic progression of length $b$. When this holds, we say that $(D_1, D_2)$ satisfies the \textbf{difference set condition}. This is expanded in more detail in \cite[Section 3]{Kenyon--Peres: translated Cantor sets}.

Kenyon and Peres proved that, in general, the dimension in question can be expressed using the Lyapunov exponent of random products of non-negative matrices.
\begin{theorem}[{\cite[Theorem 1.2]{Kenyon--Peres: translated Cantor sets}}] \label{theorem: KP main result}
Define the $2 \times 2$ matrices $A_0, \ldots, A_{b-1}$ by
\begin{equation*}
A_i(j, k) = \# \big( \left( D_1 +i+j \right) \cap \left(D_2 + kb \right) \big) \quad \left( \text{ $0 \leq j, k \leq 1$ } \right),
\end{equation*}
where $\#$ is the number of elements. Let $\lambda$ be the Lyapunov exponent of $\{A_i\}_{i=0}^{b-1}$ with uniform weight. If no finite product of the matrices is zero, then
\begin{equation*}
\mathrm{dim}_{\mathrm{H}} \left( (K_1 + t) \cap K_2 \right)
= \frac{\lambda}{\log{b}} \quad \left( \text{$\mu$-a.e. $t$} \right).
\end{equation*}
If $A_{t_n} \cdots A_{t_1} = 0$ for some finite word, then the dimension is $0$.
\end{theorem}

\begin{definition} \label{def: degenerate pair}
The pair $(D_1, D_2)$ is said to be \textbf{degenerate} if and only if there is an index $0 \leq i \leq b-1$ such that $A_i$ is singular.
\end{definition}

When degenerate, either the dimension is $0$, or the Lyapunov exponent $\lambda$ admits a closed form. See \cite{Fan--Vebitskiy} for details.

When every $A_i$ is invertible, we define the set of M\"obius maps $\pazocal{M}(D_1, D_2) \subset \mathrm{PGL}_2(\mathbb{R})$ by $\pazocal{M}(D_1, D_2) = \{[F(A_i)]\}_{i = 0}^{b-1}$. Consider the following ``one-digit forbidden family''. For every $b \geq 5$, we can prove that such a pair is either degenerate, or does not have generalized heteroclinic connections of depth $2$.

\begin{proposition} \label{prop: toothless twins}
Let $b \geq 5$. Consider $\tau, u \in \{0, 1, \ldots, b-1\}$ and let
\begin{align*}
& D_1 = \{0, 1, \ldots, b-1\} \setminus \{ \tau \}, \\
& D_2 = \{0, 1, \ldots, b-1\} \setminus \{u\}.
\end{align*}
Then, the pair $(D_1, D_2)$ is non-degenerate if and only if both $\{\tau, u\} \cap \{0, b-1\} = \varnothing$ and $\tau + u = b-1$. When non-degenerate, the family $\pazocal{M}(D_1, D_2) \subset \mathrm{PGL}_2(\mathbb{R})$ does not have generalized heteroclinic connections of depth $2$; thus we have for $\mu$-a.e.\! $t$,
\[
\mathrm{dim}_{\mathrm{H}} \left( (K_1 + t) \cap K_2 \right)
= \frac{1}{\log b} \sum_{n = 0}^\infty \big( T^n \boldsymbol{v} \big)_0.
\]
Here, the matrices $\{A_i\}_{i = 0}^{b-1}$ are conjugate to some positive matrices by Theorem \ref{thm: main theorem 3}, and an infinite matrix $T$ and a vector $\boldsymbol{v}$ are obtained from applying Theorem \ref{thm: main theorem 1} to the positive matrices.
\end{proposition}

\begin{example} \label{example: middle fifth}
Let $b = 5$, and $D_1 = D_2 = \{0, 1, 3, 4\}$. So, we are considering the middle-fifth Cantor set $\pazocal{C}_5$. Then, the matrices $A_0, \ldots, A_4$ are, respectively,
\[
\begin{pmatrix}
4 & 0\\
2 & 1
\end{pmatrix}
, \hspace{2pt}
\begin{pmatrix}
2 & 1\\
1 & 2
\end{pmatrix}
, \hspace{2pt}
\begin{pmatrix}
1 & 2\\
2 & 1
\end{pmatrix}
, \hspace{2pt}
\begin{pmatrix}
2 & 1\\
1 & 2
\end{pmatrix}
, \hspace{2pt}
\begin{pmatrix}
1 & 2\\
0 & 4
\end{pmatrix}.
\]
Let $f_i = [F(A_i)]$ for each $0 \leq i \leq 4$. Then,
\[
f_0(x) = \frac{3x+1}{5x+7}, \hspace{10pt}
f_1(x) = \frac{x}{3}, \hspace{10pt}
f_2(x) = \frac{-x}{3}, \hspace{10pt}
f_3(x) = \frac{x}{3}, \hspace{10pt}
f_4(x) = \frac{3x-1}{-5x+7}.
\]

\smallskip
Using our method, the rigorously certified value of $\lambda$ is (displayed here to 128 decimal places)
\begin{align*}
\lambda
&= \sum_{n = 0}^\infty \big( T^n \boldsymbol{v} \big)_0 \\
&= 1.1593579553271882834721584281428914386399481041242406584945182246 \\
&\hspace{15pt}\phantom{1.} 3253975873388457890499144403753476060131337370002751984393428277\cdots.
\end{align*}
This yields, for Lebesgue-a.e.\ $t$, the displayed value of $\mathrm{dim}_{\mathrm{H}} \Big( \pazocal{C}_5 \cap (\pazocal{C}_5 + t) \Big) = \frac{\lambda}{\log{5}}$ in the previous subsection. For this example, we included in Appendix \ref{appendix: comparison} a detailed comparison with the transfer-operator method by Pollicott \cite{Pollicott10} and Jurga--Morris \cite{Jurga--Morris}.
\end{example}

\begin{example}
Let $b = 4$ and $D_1 = \{0, 1, 3\}, D_2 = \{0, 2, 3\}$. Then, $(D_1, D_2)$ is non-degenerate. However, the fixed point equation for $[F(A_2)]$ is $(x+1)^2 = 0$, implying the existence of generalized heteroclinic connections of depth $2$ (as $[F(A_2)]$ is parabolic). The same conclusion holds for $D_1 = \{0, 2, 3\},$ $D_2 = \{0, 1, 3\}$. Since these examples do not satisfy the difference set condition, they are the only pairs in the one-digit forbidden family that are still beyond our reach.
\end{example}

One can ask the following natural question.
\begin{question} \label{question: classification for the intersection of translated cantor sets}
When is the pair $(D_1, D_2)$ degenerate? If non-degenerate, when does it have a generalized heteroclinic connection of depth $2$?
\end{question}
For ``thick'' Cantor sets, in particular if at least about $80\%$ of the digits are allowed for each digit set, then we can prove that the pair is either degenerate or admits the series representation, provided $b$ is large enough.

\begin{proposition} \label{prop: thick Cantor sets are Kernel expandable}
Let $b \geq 3$. Let $\rho(b) = \lfloor \frac{b-3}{5} \rfloor$, and suppose $D_1, D_2 \subset \{0, 1, \ldots, b-1\}$ satisfy $\#D_j \geq b - \rho(b)$ for $j = 1,2$. Then, the pair $(D_1, D_2)$ is either degenerate, or else, the associated M\"obius maps $\pazocal{M}(D_1, D_2)$ admit an invariant arc. If non-degenerate, we have for $\mu$-a.e.\! $t$,
\[
\mathrm{dim}_{\mathrm{H}} \left( (K_1 + t) \cap K_2 \right)
= \frac{1}{\log b} \sum_{n = 0}^\infty \big( T^n \boldsymbol{v} \big)_0.
\]
Here, the matrices $\{A_i\}_{i = 0}^{b-1}$ are conjugate to some positive matrices by Theorem \ref{thm: main theorem 3}, and an infinite matrix $T$ and a vector $\boldsymbol{v}$ are obtained from applying Theorem \ref{thm: main theorem 1} to the positive matrices.
\end{proposition}

\subsection{Application \RN{2}: Three-term recurrences}

Random Fibonacci sequences are classically defined by
\[
f_1 = f_2 = 1, \quad f_n = \pm f_{n-1} \pm f_{n-2},
\]
where the two signs are chosen to be $+$ or $-$ uniformly and independently at every step. Viswanath \cite{Viswanath} identified the almost sure growth rate of $(f_n)_n$ to be the following number.
\[
\sqrt[n]{ \left| f_n \right| } \to 1.13198824\ldots. \quad \left( \text{almost surely as }n \to \infty \right).
\]

Consider the random three-term recurrence, defined by
\begin{equation} \label{eq: three-term recurrence definition}
x_0 = x_1 = 1, \quad x_{n+1} = a_{i_{n+1}} x_n + b_{i_{n+1}} x_{n-1},
\end{equation}
where at each step a pair of positive numbers $(a_i, b_i)$ is sampled i.i.d.\ from a finite set $\{(a_1, b_1),$\\ $\ldots, (a_m, b_m)\}$ with law $\{ w_i \}_{i=1}^m$. The recurrence relation can be expressed using matrix products, as
\[
M_i = 
\begin{pmatrix}
0 & 1 \\
b_i & a_i
\end{pmatrix}, \qquad
\begin{pmatrix}
x_n \\
x_{n+1}
\end{pmatrix}
=
M_{i_{n+1}}
\begin{pmatrix}
x_{n-1} \\
x_n
\end{pmatrix}.
\]
Then, the almost sure growth rate of $(x_n)_n$ is $e^\lambda$, where $\lambda$ is the top Lyapunov exponent of $\{M_i\}_{i = 1}^m$ with weight $(w_i)_{i = 1}^m$. The following is an immediate corollary of Theorem \ref{thm: main theorem 1}.

\begin{proposition} \label{prop: three-term recurrence and Neumann series method}
Let $\{(a_i, b_i)\}_{i=1}^m \subset (0, \infty)^2$ be a finite set, $\{w_i\}_{i=1}^m$ a probability vector, and consider the sequence $\{x_n\}_{n=0}^\infty$ defined by equation \eqref{eq: three-term recurrence definition} with law $(w_i)_{i=1}^m$. Then, the matrices $\{M_i M_j\}_{i,j = 1}^m$ are positive, and
\[
\lim_{n \to \infty} \sqrt[n]{ x_n } = \exp \left( \frac12 \sum_{k = 0}^\infty \big( T^k \boldsymbol{v} \big)_0 \right) \quad \text{almost surely.}
\]
Here, $T$ and $\boldsymbol{v}$ are obtained from applying Theorem \ref{thm: main theorem 1} to $\{M_i M_j\}_{i,j = 1}^m$ with weights $(w_i w_j)_{i,j = 1}^m$.
\end{proposition}
\begin{remark}
One can also prove that $\{M_i\}_{i = 1}^m$ do not have generalized heteroclinic connections of depth $2$ with easy calculations. Also, if $a_i = 0$, then $ \begin{pmatrix} 0 & 1 \\ b_i & 0 \end{pmatrix} \begin{pmatrix} 0 & 1 \\ b_i & 0 \end{pmatrix} = b_i \begin{pmatrix} 1 & 0 \\ 0 & 1 \end{pmatrix}$, yielding a generalized heteroclinic connection of depth $2$. If $b_i = 0$, the matrix $M_i$ is singular.
\end{remark}

\section{Kernel expansion, Definitions, and Basic Lemmas} \label{section: definitions and basic properties}

\subsection{Heuristic explanation} \label{subsection: Heuristic explanation}
The core idea of Theorem \ref{thm: main theorem 1} is our novel approach, ``Kernel expansion''. Let $\{A_i\}_{i \in I}$ be a finite family of invertible positive $2 \times 2$ matrices. Consider its transposed system $\{A_i^{\top}\}_{i \in I}$, which has the same Lyapunov exponent as the original system, and fix a $\mu$-stationary measure $\widetilde{\nu}$ for the family $\{f_i^{\top}\}_{i \in I}$, where $f_i^{\top} = [F(A_i)]^{\top} = [F(A_i^{\top})]$. Since $A_i^{\top}$ is a positive matrix we have $f_i^{\top}([-1,1]) \subset (-1,1)$. Thus, we can restrict ourselves to $[-r,r]$ with some $0<r<1$, and the support of $\widetilde{\nu}$ is contained in $[-r,r]$. We define for each natural number $n$,
\[
\phi_n(x) = x^n - \sum_{i \in I} w_i \, {\big( f_i^{\top}(x) \big)}^n.
\]
By the definition of stationary measures we have
\[
\int_{[-r,r]} \phi_n \, d\widetilde{\nu}(x) = 0.
\]
So, $\{ \phi_n \}_{n \in \mathbb N}$ are ``Kernel functions'' that integrate to $0$ with respect to $\widetilde{\nu}$. If $b_{k,n}$ are as in Theorem \ref{thm: main theorem 1}, then the following formal expansion is uniformly convergent for every $n\geq 1$.
\[
\sum_{i \in I} w_i \, \Big( f_i^{\top}(x) \Big)^n = \sum_{k = 0}^\infty b_{k, n} x^k.
\]

Now, define
\begin{equation*}
\mathscr{N} := \left\{ h: [-r, r] \to \mathbb{R}
\setcond
\text{There is a bounded $c = (c_n)_{n=0}^\infty \in \mathbb{R}^{\mathbb{N}_0}$ with } h = c_0 + \sum_{n=1}^\infty c_n \phi_n \right\},
\end{equation*}
\begin{equation*}
\mathscr{A} := \left\{ g: [-r, r] \to \mathbb{R}
\setcond
\text{There is a bounded $ a =(a_k)_{k=0}^\infty \in \mathbb{R}^{\mathbb{N}_0}$ with } g = \sum_{k=0}^\infty a_k x^k \right\}.
\end{equation*}
By $f_i^{\top}([-1,1]) \subset [-r,r]$, the series in the definition of $\mathscr{N}$ and $\mathscr{A}$ are uniformly convergent. Suppose that we have $\mathscr{N} \cong \ell^\infty(\mathbb{N}_0) \cong \mathscr{A}$ via the correspondence of coefficients. Assume further that, for $T$ in Theorem \ref{thm: main theorem 1}, the map $T: \ell^\infty(\mathbb{N}_0) \to \ell^\infty(\mathbb{N}_0)$ is well-defined and bounded. Let $\mathrm{Id}: \ell^\infty(\mathbb{N}_0) \to \ell^\infty(\mathbb{N}_0)$ denote the identity map. Then, the following composition of operators is an inclusion.
\[
\mathscr{N} \xrightarrow{\cong} \ell^\infty(\mathbb{N}_0) \xrightarrow{\mathrm{Id} - T} \ell^\infty(\mathbb{N}_0) \xrightarrow{\cong} \mathscr{A}.
\]
If we na\"ively assume furthermore that the operator norm of $T$ is less than one, then the above operator is invertible with the Neumann series:
\[
(\mathrm{Id}-T)^{-1} = \sum_{n = 0}^\infty T^n.
\]

In addition, by the definition of $\boldsymbol{v}$ (for the original family), we have the following identity.
\begin{equation*}
\sum_{i \in I} w_i \, \log \norm{
A_i^{\top} \psi(x)
}_1
=
\sum_{n = 0}^\infty \boldsymbol{v}_n \, x^n,
\end{equation*}
Thus, assuming uniform convergence and using $\int \phi_n d\widetilde{\nu} = 0$,
\begin{align*}
\int_{[-r, r]} \sum_{i \in I} w_i \, \log \norm{
\frac{1}{2} A_i^{\top}
\begin{pmatrix}
1+x \\
1-x
\end{pmatrix}
}_1 \, d\widetilde{\nu} \,
&= \, \int_{[-r, r]}
\begin{pmatrix}
1, \, \phi_1, \, \phi_2, \, \cdots
\end{pmatrix}
\sum_{n=0}^\infty T^n \,\boldsymbol{v} \, d\widetilde{\nu} \, 
= \left( \sum_{n=0}^\infty T^n \, \boldsymbol{v} \right)_0.
\end{align*}
Hence we obtain the desired expression for $\lambda$ by taking the supremum over all $\widetilde{\nu}$. $\big($We will see that the right-hand side is in fact connected to the original un-transposed system.$\big)$

However, $T$ generally does not satisfy $\| T \| < 1$. So, in the proof of Theorem \ref{thm: main theorem 1}, we compute $\sum_{k = 0}^{n-1}( T^k \boldsymbol{v} )_0$ explicitly and show that it converges to $\lambda$ using a contraction/fixed-point argument.

\subsection{Definitions and Basic Lemmas}

We introduce the following classification of M\"obius maps. Since we deal with orientation-reversing maps, we need the unusual ``involution'' class. For a M\"obius map $f \in \mathrm{PGL}_2(\mathbb{R})$ satisfying $f(\widehat{\infty}) = \widehat{\infty}$, define its derivative at $\widehat{\infty}$ by $f'(\widehat{\infty}) := \left. \frac{d}{dz} \frac{1}{f(1/z)} \right|_{z = 0}$.
\begin{definition}
A fixed point $a$ of a M\"obius map $f \in \mathrm{PGL}_2(\mathbb{R})$ is said to be attracting if $|f'(a)| < 1$, repelling if $|f'(a)| > 1$, and neutral if $|f'(a)| = 1$. A non-identity M\"obius map with real coefficients is said to be:
\begin{enumerate}
\item \textbf{elliptic} if there is no fixed point in $\widehat{\mathbb{R}}$,
\item \textbf{parabolic} if there is a unique, neutral fixed point in $\widehat{\mathbb{R}}$,
\item \textbf{involution} if there are exactly two distinct neutral fixed points in $\widehat{\mathbb{R}}$, and
\item \textbf{hyperbolic} if there are exactly two distinct fixed points in $\widehat{\mathbb{R}}$, one attracting and the other repelling.
\end{enumerate}
\end{definition}

\begin{lemma} \label{lemma: Mobius classification}
A non-identity M\"obius map $f \in \mathrm{PGL}_2(\mathbb{R})$ is either elliptic, parabolic, involution, or hyperbolic. An involution M\"obius map $f$ satisfies $f^2 = \mathrm{id}$, where $\mathrm{id}: \widehat{\mathbb{R}} \to \widehat{\mathbb{R}}$ is the identity map.
\end{lemma}
The proof of Lemma \ref{lemma: Mobius classification} is elementary, and is included in Appendix \ref{appendix: proof of minor lemmas}.

Let us now give the precise definition of generalized heteroclinic connections of depth $k$.

\begin{definition} \label{def: generalized depth2 heteroclinic connection}
Let $\{A_i\}_{i \in I}$ be a finite family of invertible non-negative $2 \times 2$ matrices, and let $\{f_i\}_{i \in I} \subset \mathrm{PGL}_2(\mathbb{R})$ be the corresponding projective actions. For $k \in \mathbb{N} \cup \{\infty\}$, let
\begin{gather*}
\pazocal{F}_k = \bigcup_{n = 1}^k \, \, \big\{ f_{i_n} \circ \cdots \circ f_{i_1} \, | \, (i_1, \ldots, i_n) \in I^n \big\}, \\
\mathrm{Attr}_k = \{ a \in \widehat{\mathbb{R}} \, | \, \text{There is $f \in \pazocal{F}_k$ with $f(a) = a$ and $|f'(a)| \leq 1$} \}, \\
\mathrm{Rep}_k = \{ b \in \widehat{\mathbb{R}} \, | \, \text{There is $f \in \pazocal{F}_k$ with $f(b) = b$ and $|f'(b)| \geq 1$} \}.
\end{gather*}
Then $\{A_i\}_{i \in I}$ is said to have a \textbf{generalized heteroclinic connection of depth} $\boldsymbol{k}$ if there is $h \in \{\mathrm{id}\} \cup \pazocal{F}_k$ and $a \in \mathrm{Attr}_k$ such that $h(a) \in \mathrm{Rep}_k$.
\end{definition}

\begin{remark} \label{remark: generalized heteroclinic connections}
\, \\[-20pt]
\begin{enumerate}
\item $\mathrm{Attr}_k$ is the set of attracting or neutral fixed points, and $\mathrm{Rep}_k$ is the set of repelling or neutral fixed points. For hyperbolic elements, an attracting fixed point on the projective space corresponds to an unstable direction for a matrix, and repelling to stable.
\item We will prove that the existence of generalized heteroclinic connections of depth $\infty$ is equivalent to the existence of generalized heteroclinic connections of depth $2$. (Proposition \ref{prop: heteroclinic connection and invariant arc})
\item The notion of heteroclinic connection was introduced by Avila--Bochi--Yoccoz \cite{Avila--Bochi--Yoccoz} for finite-valued $\mathrm{SL}_2(\mathbb{R})$ cocycles. Generalized heteroclinic connections of depth $\infty$ are the analogue of the degeneracies that occur at the boundary of a connected component of the uniformly hyperbolic locus: such a point either (i) contains a projective identity, (ii) contains a parabolic element, or (iii) has a heteroclinic connection. See \cite[Theorem 4.1]{Avila--Bochi--Yoccoz} for details.
\item Avila--Bochi--Yoccoz \cite{Avila--Bochi--Yoccoz} also introduces the notion of \emph
{multicone}: a finite union of open intervals $M \subset \widehat{\mathbb{R}}$ is said to be a multicone of $\{f_i\}_{i \in I}$ if the closure of $f_i(M)$ is contained in $M$ for every $i \in I$. Then, an invariant arc is the closure of a multicone with a single connected component. Uniform hyperbolicity corresponds to the existence of a multicone, although not necessarily connected. \cite[Theorem 2.2]{Avila--Bochi--Yoccoz}
\end{enumerate}
\end{remark}

Let $\mathbb{D}_t = \left\{ z \in \mathbb{C} \, : \, |z| < t \right\}$ for $t > 0$. The following lemma states that the behavior of a real M\"obius map on $[-1,1]$ controls its behavior on $\overline{\mathbb{D}_1}$.
\begin{lemma} \label{lemma: real contraction implies complex contraction}
Let $f \in \mathrm{PGL}_2(\mathbb{R})$ be a M\"obius transformation with real coefficients, and suppose that there is $0 < r \leq 1$ with $f([-1,1]) \subset [-r, r]$. Then, we have $f(\overline{ \mathbb{D}_1} ) \subset \overline{\mathbb{D}_r}$ and $f(\mathbb{D}_1) \subset \mathbb{D}_r$.
\end{lemma}

\begin{proof}
A M\"obius transformation sends a circle to a line if and only if the circle contains the pole $f^{-1}(\widehat{\infty})$, and otherwise a circle is sent to a circle. Since the unique pole $f^{-1}(\widehat{\infty})$ of $f$ is in $\widehat{\mathbb{R}} \setminus [-1,1]$, we know that $f( \partial \mathbb{D}_1)$ is a circle. Also, $f$ is conformal since it is holomorphic on a neighborhood of $\overline{\mathbb{D}_1}$ and has non-vanishing derivative. Thus, the circle $f( \partial \mathbb{D}_1)$ crosses $\mathbb{R}$ orthogonally at $f(\pm1)$, and the center of the circle is $c = \frac{1}{2} \big( f(-1) + f(1) \big) \in [-r,r]$, radius $\rho = \frac{1}{2}\left| f(-1) - f(1) \right|$. Then, 
\[
\left| c + \rho e^{i\theta} \right| \leq |c| + \rho = \max \{ |f(-1)|, |f(1)| \} \leq r.
\]
This implies that $f(z) \in \overline{\mathbb{D}_r}$ for any $z \in \partial \mathbb{D}_1$. By the maximum modulus principle, $|f|$ attains a maximum at the boundary of $\mathbb{D}_1$. We conclude that $f( \overline{\mathbb{D}_1} ) \subset \overline{\mathbb{D}_r}$. The claim $f(\mathbb{D}_1) \subset \mathbb{D}_r$ follows by the open mapping theorem.
\end{proof}

\begin{definition}
For $x, y \in (-1, 1)$ define
\[
d_{\mathrm{hyp}} ( x, y ) = 2 \mathrm{artanh} \left( \left| \frac{ x - y }{ 1- xy } \right| \right).
\qquad \left( \mathrm{artanh}(x) = \frac{1}{2} \log \frac{1+x}{1-x}. \right)
\]
This $d_{\mathrm{hyp}}$ defines a distance on $(-1,1)$, and is called the \textbf{hyperbolic distance}.
\end{definition}

We also have the following integral representation:
\begin{equation} \label{eq: integral representation of hyperbolic metric}
d_{\mathrm{hyp}}(x, y) = \left| \int_x^y \frac{2}{1 - t^2} dt \right|.
\end{equation}

If a real M\"obius map sends the interval $[-1,1]$ strictly inside itself, it is a contraction with respect to the hyperbolic metric.
\begin{lemma} \label{lemma: small image implies contraction}
Let $f \in \mathrm{PGL}_2(\mathbb{R})$ be a M\"obius transformation with real coefficients satisfying $f([-1,1]) \subset [-r,r]$ for some $0<r<1$. Then for any $x,y \in (-1,1)$,
\[
d_{\mathrm{hyp}}(f(x),f(y)) \; \leq \; r \, d_{\mathrm{hyp}}(x,y),
\]
i.e.\ the restriction $f|_{(-1,1)}: (-1,1) \to (-1,1)$ is a strict contraction with respect to $d_{\mathrm{hyp}}$.
\end{lemma}

\begin{proof}
Suppose $0 < r < 1$ satisfies $f( [-1, 1] ) \subset [-r, r]$. By Lemma \ref{lemma: real contraction implies complex contraction}, it follows that $f(\mathbb{D}_1)\subset \mathbb{D}_r$. Define $F:\mathbb{D}_1\to\mathbb{C}$ by
\[
F(z) = \frac{1}{r}f(z).
\]
Then $F$ is holomorphic on $\mathbb{D}_1$ and satisfies $F(\mathbb{D}_1)\subset \mathbb{D}_1$. By the Schwarz--Pick theorem, for every $z\in\mathbb{D}_1$,
\[
\frac{|F'(z)|}{1-|F(z)|^2}\leq \frac{1}{1-|z|^2}.
\]
Restricting to real $x\in(-1,1)$, we obtain
\[
\frac{|f'(x)|}{1-|f(x)|^2}
=\frac{r|F'(x)|}{1-r^2|F(x)|^2}
\leq \frac{r|F'(x)|}{1-|F(x)|^2}
\leq \frac{r}{1-x^2}.
\]

Now let $x,y\in(-1,1)$. Using the integral representation \eqref{eq: integral representation of hyperbolic metric}, we have
\begin{align*}
d_{\mathrm{hyp}}(f(x),f(y))
&=\left| \int_{f(y)}^{f(x)} \frac{2}{1-u^2} \, du \right|
=\left| \int_y^x \frac{2f'(t)}{1-f(t)^2} \, dt \right|\\
&\leq \int_{\min\{x,y\}}^{\max\{x,y\}} \frac{2|f'(t)|}{1-f(t)^2} \, dt
\leq r \int_{\min\{x,y\}}^{\max\{x,y\}} \frac{2}{1-t^2} \, dt
= r \, d_{\mathrm{hyp}}(x,y).
\end{align*}
Since $0<r<1$, this shows that $f|_{(-1,1)}$ is a strict contraction with respect to $d_{\mathrm{hyp}}$.
\end{proof}

\begin{definition}
Denote by $\mathscr{P}(X)$ the set of Borel probability measures on $X$. For $\nu_1, \nu_2 \in \mathscr{P}(-1,1)$, define the set of couplings as
\[
\Pi( \nu_1, \, \nu_2) = \left\{ \eta \in \mathscr{P} \Big( (-1, 1) \times (-1, 1) \Big) \setcond (\pi_1)_* \eta = \nu_1, (\pi_2)_* \eta = \nu_2 \right\}.
\]
Here, $\pi_j$ is the projection onto the $j$-th coordinate for $j = 1,2$. Define
\[
\mathscr{P}_1(-1,1) = \left\{ \nu \in \mathscr{P}(-1,1) \setcond \int_{(-1,1)} d_{\mathrm{hyp}}(x,0) d\nu(x) < \infty \right\}.
\]
We define the \textbf{Wasserstein-1 metric} $W_1^{d_{\mathrm{hyp}}}$ on $\mathscr{P}_1(-1,1) \times \mathscr{P}_1(-1,1)$ by
\[
W_1^{d_{\mathrm{hyp}}}(\nu_1, \nu_2) = \inf_{\eta \in \Pi(\nu_1, \nu_2)} \int_{(-1, 1) \times (-1, 1)} d_{\mathrm{hyp}}(x, y) d\eta(x, y). \qquad \big( \nu_1, \nu_2 \in \mathscr{P}_1(-1,1). \big)
\]
\end{definition}

It is known that $W_1^{d_{\mathrm{hyp}}}(\nu_1, \nu_2)$ is finite for $\nu_1, \nu_2 \in \mathscr{P}_1(-1,1)$, and the completeness of $\left( \mathscr{P}_1(-1,1), W_1^{d_{\mathrm{hyp}}} \right)$ follows from the completeness of $\left( (-1,1), d_{\mathrm{hyp}} \right)$ \cite[Theorem 6.18]{Villani}. In this paper we write $W_1$ for $W_1^{d_{\mathrm{hyp}}}$.

\section{Proof of Theorem \ref{thm: main theorem 1}} \label{section: proof of main theorem 1}

\begin{proof}[Proof of Theorem \ref{thm: main theorem 1}]

Let $F_i = F(A_i)$ and $f_i = [F_i] \in \mathrm{PGL}_2(\mathbb{R})$. Then, by the positivity of $A_i$, there is $0 < r < 1$ such that $f_i([-1,1]) \subset [-r, r]$ for all $i \in I$. Fix such $r$.

Let $g \in \mathrm{PGL}_2(\mathbb{R})$ be any M\"obius transformation. Define $s: \widehat{\mathbb{C}} \to \widehat{\mathbb{C}}$ by $s(x) = -\frac{1}{x}$ for $x \in \mathbb{C}$ and $s(0) = \widehat{\infty}$, $s(\widehat{\infty}) = 0$. Then, we have
\begin{equation} \label{eq: transpose identity}
g^{\top}
= s \circ g^{-1} \circ s.
\end{equation}
Suppose $g([-1,1]) \subset (-1,1)$. Since $g: \widehat{\mathbb{R}} \to \widehat{\mathbb{R}}$ is a homeomorphism, we have
\[
g^{\top}\big([-1,1]\big) = s \bigg( g^{-1} \big( \widehat{\mathbb{R}} \setminus (-1,1) \big) \bigg) \subset s \bigg( \widehat{\mathbb{R}} \setminus [-1,1] \bigg) = (-1,1).
\]
This shows that we have the following for every $i \in I$.
\begin{equation} \label{eq: bounding the transpose of fi}
f_i^{\top}([-1,1]) \subset (-1,1).
\end{equation}

Write
\begin{equation*}
F_i=
\begin{pmatrix}
\alpha_i & \beta_i \\
\gamma_i & \delta_i
\end{pmatrix}.
\end{equation*}

\begin{lemma} \label{lemma: denominator is positive}
Suppose $A$ is an invertible non-negative $2 \times 2$ matrix. Let $F(A) = \begin{pmatrix} a & b \\ c & d \end{pmatrix}$. Then $cx + d > 0$ for any $x \in [-1,1]$.
\end{lemma}

\begin{proof}
Write $A=\begin{pmatrix} p & q \\ r & s \end{pmatrix}$ with $p, q, r, s \geq 0$.
By the definition of $F$ in equation \eqref{eq: definition of F}, we have $|c|=|\tfrac12(p-q+r-s)| \leq \tfrac12(p+q+r+s) = d$, and equality implies that $A$ is singular. Hence $cx+d \geq d - |c| > 0$ for any $x \in [-1,1]$.
\end{proof}

Take the principal branch of $\mathrm{Log}$ with $\mathrm{Log}(1) = 0$. By Lemma \ref{lemma: denominator is positive}, we have $\delta_i > 0$. Since $\left| f_i^{\top}(0) \right| = \left| \frac{ \gamma_i }{ \delta_i } \right| < 1$ by equation \eqref{eq: bounding the transpose of fi}, we can define the following map, which is holomorphic on the neighborhood of $\overline{\mathbb{D}_1}$.
\[
\ell_i(z) = \mathrm{Log} \big( \gamma_i z + \delta_i \big). \quad \big( z \in \overline{\mathbb{D}_1}, \, i \in I \big).
\]
Then $\ell_i(x) = \log \norm{ A_i \psi(x) }_1$ for $x\in [-1,1]$. $\big($We remark that the proof of Theorem \ref{thm: main theorem 1} itself only uses real values; the complex extension is needed later.$\big)$

\medskip

Denote by $\pazocal{H}(D)$ the set of holomorphic functions on a domain $D$. We define
\[
\mathscr L: \pazocal{H}(\mathbb{D}_1) \to \pazocal{H}(\mathbb{D}_1),
\qquad
\mathscr Lh(w)  =\sum_{i\in I} w_i \Big( h\circ f_i^\top(w) - h(0) \Big).
\]
Then, the following lemma shows that the infinite matrix $T$ is the representation matrix of $\mathscr L$ with respect to the monomial basis $(x^n)_{n=0}^\infty$. For $h\in \pazocal H(\mathbb D_1)$ with power series $h(w) = \sum_{n = 0}^\infty a_n w^n$, write $v(h) = (a_n)_{n=0}^\infty$ to denote its coefficient vector.

\begin{lemma} \label{lemma: T is the representation matrix of L}
Suppose $h\in \pazocal{H}(\mathbb{D}_1)$ has bounded Taylor coefficients on $\mathbb{D}_1$. Then,
\[
T v(h) = v\big( \mathscr Lh \big).
\]
\end{lemma}

\begin{proof}[Proof of Lemma \ref{lemma: T is the representation matrix of L}]
As noted in the heuristic subsection \ref{subsection: Heuristic explanation}, we have for every $n\geq 1$,
\[
g_n(x) \;\coloneqq\; \sum_{i \in I} w_i \, \Big( f_i^{\top}(x) \Big)^n \;=\; \sum_{k = 0}^\infty b_{k, n} x^k.
\]
By equation \eqref{eq: bounding the transpose of fi} and Lemma \ref{lemma: real contraction implies complex contraction}, there is $0<\rho<1$ with $\max_{z\in \overline{\mathbb{D}_1}} \big|f_i^\top(z)\big| \leq \rho$ for every $i\in I$. Since every $f_i$ is holomorphic on a neighborhood of $\overline{\mathbb{D}_1}$, there are $R>1$ and $0<q<1$ such that every $f_i$ is holomorphic on $\overline{\mathbb{D}_R}$ and $\max_{z\in \overline{\mathbb{D}_R}} \big|f_i^\top(z)\big| \leq q$. Then, for every $k\geq0$ and $n\geq 1$,
\[
\big| b_{k,n} \big|
\leq
\frac{1}{2\pi} \left| \oint_{|z|=R} \frac{g_n(z)}{z^{k+1}}dz \right|
\leq
\frac{q^n}{R^k}.
\]

Next, let $v(h) = (a_n)_{n=0}^\infty$. By assumption, we have $v(h)\in \ell^\infty(\mathbb N_0)$. Denote by $\|\cdot\|_\infty$ the $\ell^\infty$-norm on $\ell^\infty(\mathbb N_0)$. Then, for any $x\in[-1,1]$,
\[
\sum_{k=0}^\infty \sum_{n=1}^\infty \big| b_{k,n} a_n x^k \big|
\;\leq\;
\|v(h)\|_\infty \sum_{k=0}^\infty \sum_{n=1}^\infty \frac{q^n}{R^k}
\;<\;
\infty.
\]
Therefore, we can rearrange the following double sum, and obtain $T v(h) = v\big( \mathscr Lh \big)$:
\begin{align*}
\sum_{k=0}^\infty \big(Tv(h)\big)_k \; x^k
&=
\sum_{k=0}^\infty \left( \sum_{n=1}^\infty b_{k,n} a_n \right) x^k
=
\sum_{n=1}^\infty a_n \left( \sum_{k=0}^\infty b_{k,n} x^k \right) \\
&=
\sum_{i\in I} w_i \sum_{n=1}^\infty a_n \Big( f_i^\top(x) \Big)^n
=
\sum_{i\in I} w_i \Big( h\circ f_i^\top(x) - h(0) \Big)
=
\mathscr Lh(x). \qedhere
\end{align*}
\end{proof}

Take any $x \in \overline{\mathbb{D}_1}$ and $c \in \mathbb{C}$, and consider $u(w) = c + \log(1+xw)$. We have
\[
\mathscr Lu(w)
=
\sum_{i\in I} w_i \mathrm{Log} \Big(1 + x f_i^\top(w) \Big).
\]
Now, notice the following identity for the M\"obius map $f_i$, which follows by direct calculation.
\[
1 + x f_i^\top(w) = \frac{\gamma_i x + \delta_i}{\delta_i} \cdot \frac{1+f_i(x)w}{1+f_i(0)w}.
\]
Therefore,
\[
\mathscr Lu(w)
=
\sum_{i\in I} w_i \left\{
\ell_i(x) - \ell_i(0) + \mathrm{Log} \Big(1 + f_i(x)w \Big) - \mathrm{Log} \Big(1 + f_i(0)w \Big)
\right\}.
\]
By taking the coefficient vector and applying Lemma \ref{lemma: T is the representation matrix of L}, we obtain the following coboundary identity.
\begin{equation} \label{eq: Tv}
T v(x \, ; \, c) = \sum_{i \in I} w_i \left\{ v \bigg( f_i(x) \, ; \, \ell_i(x) \bigg) - \, v \bigg( f_i(0) \, ; \, \ell_i(0) \bigg) \right\}.
\end{equation}

Since $f_i([-1,1]) \subset [-r, r]$, we have $f_i(\overline{\mathbb{D}_1}) \subset \overline{\mathbb{D}_r}$ by Lemma \ref{lemma: real contraction implies complex contraction}. Hence $|f_i(x)| \leq r <1$. Thus, we have $T(V) \subset V$ as promised.

\smallskip
Let $\boldsymbol{v} = \sum_{i \in I} w_i v \bigg( f_i(0) \, ; \, \ell_i(0) \bigg)$, which matches the definition in the statement of Theorem \ref{thm: main theorem 1}.
For $n \geq 1$ and a word $\lbar{i}=(i_1,\dots,i_n)\in I^n$, set
\[
w_{\lbar{i}} = w_{i_1} \cdots w_{i_n},
\qquad
f_{\lbar{i}} = f_{i_n}\circ\cdots\circ f_{i_1},
\qquad
\lbar{i}^* = (i_1,\dots,i_{n-1})\in I^{n-1}.
\]
For $n=1$ we interpret as $\lbar{i}^*=\varnothing$, and $f_{\varnothing}$ is defined as the identity map.
\begin{claim}
For any $n \geq 1$,
\begin{equation} \label{eq: finite sum of T}
\sum_{j=0}^{n-1} T^j \boldsymbol{v} = \sum_{\lbar{i} \in I^n} w_{\lbar{i}} v \bigg( f_{\lbar{i}}(0); \ell_{i_n} ( f_{\lbar{i}^*}(0) ) \bigg).
\end{equation}
\end{claim}

\begin{proof}
We prove by induction on $n$. Denote the right-hand side of the equation \eqref{eq: finite sum of T} as $S_n$. For $n = 1$ the equality follows from the definition of $\boldsymbol{v}$. Assume the claim is true for $n$, then
\[
\sum_{j=0}^{n} T^j \boldsymbol{v}
=\boldsymbol{v} + T \sum_{j=0}^{n-1} T^j \boldsymbol{v}
=\boldsymbol{v} + T S_n.
\]
Using linearity of $T$ and equation \eqref{eq: Tv}, we obtain
\begin{align*}
T S_n
&=\sum_{\lbar{i}\in I^n} w_{\lbar{i}}\,
T\, v\!\bigg( f_{\lbar{i}}(0)\,;\, \ell_{i_n}\!\big( f_{\lbar{i}^*}(0)\big)\bigg)
=\sum_{\lbar{i}\in I^n} w_{\lbar{i}}
\sum_{k\in I} w_k
\bigg\{
v\big(f_k(f_{\lbar{i}}(0))\,;\,\ell_k(f_{\lbar{i}}(0))\big)-v\big(f_k(0)\,;\,\ell_k(0)\big)
\bigg\}\\
&=\sum_{\lbar{i}\in I^n}\sum_{k\in I} w_{\lbar{i}}w_k\,
v\big(f_k\circ f_{\lbar{i}}(0)\,;\,\ell_k(f_{\lbar{i}}(0))\big)
\;-\;\sum_{\lbar{i}\in I^n} w_{\lbar{i}}
\sum_{k\in I} w_k\, v\big(f_k(0)\,;\,\ell_k(0)\big).
\end{align*}
Since $\sum_{\lbar{i}\in I^n} w_{\lbar{i}}=(\sum_{i\in I} w_i)^n=1$, the second term equals $-\boldsymbol{v}$.
Therefore,
\begin{equation*}
\sum_{j=0}^{n} T^j \boldsymbol{v}
=\boldsymbol{v}+T S_n
=\sum_{\lbar{i}\in I^n}\sum_{k\in I} w_{\lbar{i}}w_k\,
v\big(f_k\circ f_{\lbar{i}}(0)\,;\,\ell_k(f_{\lbar{i}}(0))\big)
= S_{n+1}. \qedhere
\end{equation*}
\end{proof}

We now prove that $ \sum_{j=0}^\infty \left( T^j \boldsymbol{v} \right)_0$ converges. First, define $\mathscr{H}: \mathscr{P}_1(-1,1) \to \mathscr{P}(-1,1)$ by
\[
\mathscr{H}\nu = \sum_{i \in I} w_i \, ( f_i )_* \nu \quad \Big( \nu \in \mathscr{P}_1(-1,1) \Big).
\]
By Lemma \ref{lemma: small image implies contraction}, for every $i \in I$ and $x,y \in (-1, 1)$
\begin{equation} \label{eq: contraction in hyperbolic metric in the proof of main theorem}
d_{\mathrm{hyp}}( f_i(x), f_i(y) ) \leq r \, d_{\mathrm{hyp}}(x, y).
\end{equation}

\begin{claim}
For any $ \nu \in \mathscr{P}_1(-1,1)$ we have $\mathscr{H} \nu \in \mathscr P_1(-1,1)$.
\end{claim}

\begin{proof}
Let $x \in (-1,1)$, and $\nu \in \mathscr{P}_1(-1,1)$. Using the triangle inequality and equation \eqref{eq: contraction in hyperbolic metric in the proof of main theorem},
\[
d_{\mathrm{hyp}}\big( f_i(x),0 \big) \leq d_{\mathrm{hyp}}\big( f_i(x),f_i(0) \big) + d_{\mathrm{hyp}}\big( f_i(0),0 \big)
\leq r \, d_{\mathrm{hyp}}(x,0)+d_{\mathrm{hyp}}\big( f_i(0),0 \big).
\]
Integrating against $\nu$ and summing with weights $w_i$ yields
\begin{align*}
\int d_{\mathrm{hyp}}(x,0) \, d(\mathscr{H}\nu)
=\sum_{i\in I} w_i \int d_{\mathrm{hyp}}\big( f_i(x),0 \big) \, d\nu \leq r \int d_{\mathrm{hyp}}(x,0)\, d\nu+\sum_{i\in I} w_i \, d_{\mathrm{hyp}}\big( f_i(0),0 \big).
\end{align*}
Since $f_i(0)\in[-r,r]$, each $d_{\mathrm{hyp}}(f_i(0),0)$ is finite, and by $\nu \in \mathscr{P}_1(-1,1)$, the last expression is finite. We conclude $\mathscr{H}\nu \in \mathscr{P}_1(-1,1)$.
\end{proof}

\begin{claim} \label{claim: H is a contraction}
The map $\mathscr{H}$ is a contraction with respect to the Wasserstein-1 metric.
\end{claim}

\begin{proof}[Proof of Claim]
Take any $\nu_1, \nu_2 \in \mathscr{P}_1(-1,1)$. Take any coupling $\eta \in \Pi(\nu_1, \nu_2)$. We define $\Phi_i: (-1,1) \times (-1,1) \to (-1,1) \times (-1,1)$ by $\Phi_i( x,y ) = \Big( f_i(x), f_i(y) \Big)$ and let $\tau_i = (\Phi_i)_* \eta$. Let $\pi_j$ be the projection onto the $j$-th coordinate for $j = 1,2$. Since $\pi_j \circ \Phi_i = f_i \circ \pi_j$, we have $(\pi_j)_* \tau_i = (f_i)_* \nu_j$ for each $j = 1,2$. Thus for $\tau = \sum_{i \in I} w_i \tau_i$, we have $\tau \in \Pi( \mathscr{H} \nu_1, \mathscr{H} \nu_2 )$. Then,
\begin{align*}
&W_1( \mathscr{H}\nu_1, \mathscr{H}\nu_2 )
\leq \int d_{\mathrm{hyp}}(\xi, \zeta) d\tau(\xi, \zeta)
= \sum_{i \in I} w_i \int d_{\mathrm{hyp}}(\xi, \zeta) d\tau_i(\xi, \zeta) \\
&= \sum_{i \in I} w_i \int d_{\mathrm{hyp}}( f_i(x), f_i(y) ) d\eta(x,y)
\leq \sum_{i \in I} w_i \int r \; d_{\mathrm{hyp}}( x, y ) d\eta(x,y)
= r \int d_{\mathrm{hyp}}( x, y ) d\eta(x,y).
\end{align*}
By taking the infimum over $\eta$, we conclude that
\[
W_1( \mathscr{H}\nu_1, \mathscr{H}\nu_2 ) \leq r W_1( \nu_1, \nu_2 ). \qedhere
\]
\end{proof}

Since $\left( \mathscr{P}_1(-1,1), W_1 \right)$ is complete, by Banach's fixed point theorem there is $\nu^* \in \mathscr{P}_1(-1,1)$ such that $\mathscr{H} \nu^* = \nu^*$, and for any $\nu_0 \in \mathscr{P}_1(-1,1)$, the measure $\mathscr{H}^n \nu_0$ converges to $\nu^*$ in $W_1$ as $n \to \infty$. Also, any $\mu$-stationary measure on $[-1,1]$ must have support on $[-r,r] \subset (-1,1)$, and the only stationary measure is $\nu^*$ by the uniqueness of the fixed point of $\mathscr{H}$.

Consider the Dirac measure at $0$, which we denote by $\delta_0 \in \mathscr{P}_1(-1,1)$, and let $\nu_n = \mathscr{H}^n \delta_0$. By equation \eqref{eq: finite sum of T}, for any $n \geq 1$
\begin{align} \label{eq: finite sum of 0-th coordinate}
\sum_{j=0}^{n-1}\left( T^j \boldsymbol{v} \right)_0
=
\sum_{i \in I} w_i \int_{(-1,1)} \ell_i(x) d\nu_{n-1}(x).
\end{align}
By Lemma \ref{lemma: denominator is positive}, we have $\gamma_i x + \delta_i > 0$ for all $x \in [-1,1]$. Thus the integrand is continuous and bounded. Now, $\nu_n \to \nu^*$ in $W_1$ implies that $\nu_n$ converges weakly to $\nu^*$ by \cite[Theorem 6.9]{Villani}. Hence, as $n \to \infty$, the limit of equation \eqref{eq: finite sum of 0-th coordinate} exists, and
\begin{align} \label{eq: 0-th coordinate}
\sum_{n=0}^\infty \big( T^n \boldsymbol{v} \big)_0
= \sum_{i \in I} w_i \int_{(-1,1)} \ell_i(x) d\nu^*(x).
\end{align}
By $\ell_i(x) = \log \norm{ A_i \psi(x) }_1$, this coincides with the Furstenberg--Kifer formula in equation \eqref{eq: integral form by Furstenberg Kifer} by the uniqueness of the stationary measure.

Therefore, the series $\sum_{n=0}^\infty (T^n \boldsymbol{v})_0$ converges and equals $\lambda$.
\end{proof}

\begin{remark}
The argument above also proved that for $k \geq 1$,
\begin{align*}
\sum_{j=0}^{n-1} \left( T^j \boldsymbol{v} \right)_k
=
\frac{(-1)^{k-1}}{k} \int_{(-1,1)} x^k d\nu_n(x)
\xrightarrow[]{n \to \infty} \frac{(-1)^{k-1}}{k} \int_{(-1,1)} x^k d\nu^*(x).
\end{align*}
\end{remark}

\section{Proof of Theorem \ref{thm: main theorem 2}} \label{section: proof of main theorem 2}

\noindent\textbf{Overview of this section.}
This section gives a quantitative approximation of
\[
\lambda=\sum_{n=0}^{\infty}(T^n\boldsymbol v)_0
\]
by separating the error into two parts of different natures: the tail of the Neumann series, and the truncation error coming from replacing $T$ by a finite matrix $T_m$. The tail $\sum_{k \geq n}(T^k\boldsymbol v)_0$ is handled in Subsection \ref{subsection: Cutting the tail of the series} using the Wasserstein contraction of the operator $\mathscr{H}$. The truncation error is more delicate. Subsection~\ref{subsection: Integral representation} derives integral formulas for $T$ and $T_m$, where the cancellations become visible. Subsection \ref{subsection: Integral representation for iterates} then lifts these formulas to $T^n$ and $T_m^n$. Finally, Subsection \ref{subsection: Integral operators and telescoping} introduces the relevant holomorphic integral operators, proves their oscillation bounds, and combines them with a telescoping identity to obtain an explicit estimate for $\sum_{j=0}^{N-1}\big((T^j-T_m^j)\boldsymbol v\big)_0$.

\subsection{Cutting the tail of the series} \label{subsection: Cutting the tail of the series}

\begin{proposition} \label{prop: cutting the tail of the Neumann series}
Under the assumptions of Theorem \ref{thm: main theorem 1}, suppose $0 < r < 1$ satisfies $f_i([-1,1]) \subset [-r,r]$ for every $i \in I$. Let
\[
E = \frac{1}{ 1 - r } \left( \sum_{i \in I} w_i \frac{ \left| f_i^{\top}(0) \right|}{ 1 + \sqrt{ 1 - {\left| f_i^{\top}(0) \right|}^2 } } \right) \left( \sum_{i \in I} w_i d_{\mathrm{hyp}}\big( f_i(0), 0 \big) \right).
\]
Then, for any $n \in \mathbb N$,
\[
\Bigg| \sum_{k=0}^\infty \big( T^k \boldsymbol{v} \big)_0 - \sum_{k=0}^{n-1} \big( T^k \boldsymbol{v} \big)_0 \Bigg| \leq E r^{n-1}.
\]
\end{proposition}

We need the following lemma.
\begin{lemma} \label{lemma: hyperbolic lipschitz constant}
Let $g \in C^1(-1, 1)$. Define
\[
L_g = \sup_{x \in (-1,1)} \frac{1-x^2}{2} \left| g'(x) \right|.
\]
If $L_g$ is finite, then for any $\nu_1, \nu_2 \in \mathscr{P}_1(-1,1)$,
\[
\left| \int g d\nu_1 - \int g d\nu_2 \right| \leq L_g \, W_1(\nu_1, \nu_2).
\]
\end{lemma}

\begin{proof}
First, for any $x, y \in (-1, 1)$,
\begin{align*}
\left| g(x) - g(y) \right|
= \left| \int_x^y g'(t) dt \right|
= \left| \int_x^y \frac{1 - t^2}{2} g'(t) \frac{2 dt}{1-t^2} \right|
\leq L_g \, d_{\mathrm{hyp}}(x, y).
\end{align*}
In particular, $\int g d\nu$ exists for $\nu \in \mathscr{P}_1(-1,1)$ by letting $y = 0$. Take any coupling $\eta \in \Pi(\nu_1, \nu_2)$. We have
\[
\left| \int g d\nu_1 - \int g d\nu_2 \right|
= \left| \int \Big( g(x) - g(y) \Big) d\eta(x,y) \right|
\leq L_g \, \int d_{\mathrm{hyp}}(x, y) d\eta(x,y).
\]
Taking the infimum over $\eta$ proves the desired inequality.
\end{proof}

\begin{proof}[Proof of Proposition \ref{prop: cutting the tail of the Neumann series}]
We recall the notations: for each $i\in I$, let $F_i = F(A_i)$ and
\begin{equation*}
F_i=
\begin{pmatrix}
\alpha_i & \beta_i \\
\gamma_i & \delta_i
\end{pmatrix}, \qquad
\ell_i(x) = \log{ \big( \gamma_i x + \delta_i \big) } \quad \big( x \in (-1,1) \big).
\end{equation*}
A direct calculus computation shows that for $|a|<1$, $\sup_{x\in(-1,1)}\frac{1-x^2}{1+ax}=\frac{2}{1+\sqrt{1-a^2}}$. Hence,
\[
L_{\ell_i} = \sup_{x \in (-1,1)} \frac{1-x^2}{2} \left| \frac{\gamma_i}{\gamma_i x + \delta_i} \right|
= \frac{1}{2} \cdot \left| f_i^{\top}(0) \right| \sup_{x \in (-1,1)} \frac{ 1 - x^2 }{ 1 + \left| f_i^{\top}(0) \right| x}
= \frac{ \left| f_i^{\top}(0) \right|}{ 1 + \sqrt{ 1 - {\left| f_i^{\top}(0) \right|}^2 } }.
\]
By equation \eqref{eq: finite sum of 0-th coordinate}, \eqref{eq: 0-th coordinate}, and Lemma \ref{lemma: hyperbolic lipschitz constant},
\begin{align*}
\Bigg| \sum_{k=0}^\infty \big( T^k \boldsymbol{v} \big)_0\hspace{-2pt} - \hspace{-2pt}\sum_{k=0}^{n-1} \big( T^k \boldsymbol{v} \big)_0 \Bigg|
= \left| \sum_{i \in I} w_i \left\{ \int\hspace{-2pt}\ell_i(x) d\nu^*(x) - \hspace{-4pt}\int \hspace{-2pt}\ell_i(x) d\nu_{n-1}(x) \hspace{-2pt}\right\}\hspace{-2pt}\right|
\leq \sum_{i \in I} w_i L_{\ell_i} W_1( \nu^*, \nu_{n-1} ).
\end{align*}

Recall that $\mathscr{H}\eta = \sum_{i \in I} w_i \, ( f_i )_* \eta$ for $\eta \in \mathscr{P}_1(-1,1)$. By Claim \ref{claim: H is a contraction}, we have
\[
W_1( \mathscr{H}\nu, \mathscr{H}\nu' ) \leq r W_1( \nu, \nu'). \qquad \left( \nu, \nu' \in \mathscr{P}_1(-1,1). \right)
\]
Let $\nu^*$ be the unique fixed point of $\mathscr{H}$ obtained in the proof of Theorem \ref{thm: main theorem 1}. Recall that we defined $\nu_k = \mathscr{H}^k \delta_0$ for $k \in \mathbb{N}$, where $\delta_0$ is the Dirac measure at $0$. Since $\nu_k\to \nu^*$ in $W_1$, using triangle inequality for $W_1( \nu_m, \nu_{n-1})$ and then letting $m \to \infty$ yields
\[
W_1( \nu^*, \nu_{n-1}) \leq \sum_{k=n-1}^{\infty} W_1(\nu_{k+1}, \nu_k).
\]
Then,
\begin{align*}
W_1( \nu^*, \nu_{n-1} )
\leq \sum_{k = n-1}^\infty W_1( \mathscr{H}^k \nu_1, \mathscr{H}^k \delta_0 ) 
\leq \sum_{k = n-1}^\infty r^k W_1( \nu_1, \delta_0 )
= \frac{ r^{n-1} }{ 1 - r } W_1( \nu_1, \delta_0 ).
\end{align*}
Since $\delta_0$ is a Dirac measure, we have
$W_1(\nu_1,\delta_0)=\int d_{\mathrm{hyp}}(x,0)\,d\nu_1(x)
=\sum_{i\in I} w_i\, d_{\mathrm{hyp}}(f_i(0),0)$.

In conclusion,
\begin{align*}
\Bigg| \sum_{k=0}^\infty \big( T^k \boldsymbol{v} \big)_0 - \sum_{k=0}^{n-1} \big( T^k \boldsymbol{v} \big)_0 \Bigg|
&\leq \left( \sum_{i \in I} w_i L_{\ell_i} \right) W_1(\nu^*, \nu_{n-1}) \\
&\leq \left( \sum_{i \in I} w_i \frac{ \left| f_i^{\top}(0) \right|}{ 1 + \sqrt{ 1 - {\left|f_i^{\top}(0) \right|}^2 } } \right) \left( \sum_{i \in I} w_i d_{\mathrm{hyp}}\big( f_i(0), 0 \big) \right) \frac{ r^{n-1} }{ 1 - r }. \qedhere
\end{align*}
\end{proof}

\subsection{Integral representation} \label{subsection: Integral representation}

In this subsection, we begin to analyze the difference between $T$ and its finite approximation. Although for each fixed $k$ the series for $\big(Tv(x;c)\big)_k$ is absolutely convergent, termwise absolute bounds are generally useless for controlling the operator uniformly in $k$: the resulting bounds typically blow up as $k \to \infty$. The integral representations below capture cancellations in the oscillatory combinatorial sums.

From the infinite matrix $T = (b_{k, n})_{k, n \in \mathbb{N}_0}$ and a natural number $m$, define $T_m = (c_{k,n})_{k,n \in \mathbb{N}_0}$ by $c_{k,n} = b_{k,n}$ when $k, n \leq m-1$ and $0$ otherwise. Let $0 < r <1$ be the constant satisfying $f_i([-1,1]) \subset [-r,r]$ for all $i \in I$. We denote by $\mathrm{i}$ the square root of $-1$.

\begin{proposition} \label{prop: integral representation}
Let $x \in \mathbb{D}_r$ and $c \in \mathbb{C}$. We have for any $k \geq 1$,
\begin{equation} \label{eq: integral representation of T}
\bigg( T v( x \, ; \, c ) \bigg)_k
\;=\;
- \sum_{i \in I} w_i \frac{1}{2\pi \mathrm{i}} \oint_{|z| = r} \frac{ 1 }{\; 1 - z \;} \frac{ \left( - f_i \left( \frac{x}{ z } \right) \right)^k }{ k } dz,
\end{equation}
and for any $1 \leq k \leq m-1$,
\begin{equation} \label{eq: integral representation of Tm}
\bigg( T_m v( x \, ; \, c ) \bigg)_k
\;=\;
- \sum_{i \in I} w_i \frac{1}{2\pi \mathrm{i}} \oint_{|z| = r} \frac{ 1 - z^{m-1} }{ 1 - z } \frac{ \left( - f_i \left( \frac{x}{ z } \right) \right)^k }{ k } dz.
\end{equation}
\end{proposition}

Let us first prove the following lemma.

\begin{lemma} \label{lemma: coboundary property}
Let $h(w) = \sum_{n=0}^\infty a_n w^n \in \pazocal H(\mathbb{D}_1)$. Then, for any $w \in \mathbb{D}_r$,
\begin{align*}
&\frac{1}{ 2 \pi \mathrm{i} } \oint_{ |z| = r } \;\;\frac{1}{\;1-z\;}\;\; h \left( \frac{w}{z} \right) dz \;=\; \sum_{n= 1}^\infty a_nw^n \;=\; h(w) - h(0),
\\
&\frac{1}{ 2 \pi \mathrm{i} } \oint_{ |z| = r } \frac{1-z^{m-1}}{1-z} h \left( \frac{w}{z} \right) dz \;=\; \sum_{n=1}^{m-1}a_n w^n.
\end{align*}
\end{lemma}

\begin{proof}
Take $\rho$ so that $|w|/r < \rho < 1$. By assumption, we have a Taylor series that is uniformly convergent on $\mathbb{D}_\rho$, $h(\zeta) = \sum_{\ell = 0}^\infty a_\ell \zeta^\ell$. Then, we also have the following uniformly converging series for $|z| = r$.
\[
h\left( \frac{w}{z} \right) = \sum_{\ell = 0}^\infty a_\ell w^\ell z^{-\ell}.
\]
Since $\frac{1}{1-z} = \sum_{q = 0}^\infty z^q$ is uniformly convergent for $|z| = r$, we have a uniformly converging series,
\[
\frac{1}{1-z} h\left( \frac{w}{z} \right) = \sum_{\ell, q \geq 0} a_{\ell} w^\ell z^{q - \ell}.
\]
By integrating termwise,
\[
\frac{1}{ 2 \pi \mathrm{i} } \oint_{ |z| = r } \frac{1}{1-z} h \left( \frac{w}{z} \right) dz
= \sum_{\ell = 1}^\infty a_\ell w^\ell = h(w) - h(0).
\]
The same argument works for the second assertion.
\end{proof}

\begin{proof}[Proof of Proposition \ref{prop: integral representation}]
Fix $x\in \mathbb D_r$ and $c\in \mathbb C$. By equation \eqref{eq: Tv} and Lemma \ref{lemma: coboundary property}, for any $k \geq 1$,
\begin{equation*}
\bigg( T v(x \, ; \, c) \bigg)_k
= - \sum_{i \in I} w_i \left\{ \frac{\big( - f_i(x) \big)^k}{k} - \frac{\big( - f_i(0) \big)^k}{k} \right\}
= - \sum_{i \in I} w_i \frac{1}{ 2 \pi \mathrm{i} } \oint_{ |z| = r } \frac{1}{1-z} \frac{ \left( - f_i \left( \frac{x}{ z } \right) \right)^k }{ k } dz.
\end{equation*}
This proved the formula for $T$ in equation \eqref{eq: integral representation of T}.

Next, fix $1\leq k\leq m-1$, and define $h_i\in \pazocal H(\mathbb D_1)$ for each $i\in I$ by
\[
h_i(w) \;\coloneqq\; - \frac{\big(-f_i(w)\big)^k}{k}.
\]
Write $h_i(w) = \sum_{n=0}^\infty a_n^{(i)} w^n$. Then, by equation \eqref{eq: Tv},
\[
\sum_{i\in I} w_i \sum_{n=1}^\infty a_n^{(i)} x^n
=
\sum_{i\in I} w_i \big( h_i(x) - h_i(0) \big)
=
\Big( T v(x;c) \Big)_k
=
\sum_{n=1}^\infty b_{k,n} \left( - \frac{(-x)^n}{n} \right).
\]
Thus,
\begin{equation} \label{connecting ani and bkn}
b_{k,n} \frac{(-1)^{n+1}}{n}
=
\sum_{i\in I} w_i a_n^{(i)}.
\end{equation}
Now, using the definition of $T_m$, equation \eqref{connecting ani and bkn}, and Lemma \ref{lemma: coboundary property}, we obtain
\[
\Big( T_m v(x;c) \Big)_k
=
\sum_{n=1}^{m-1} b_{k,n} \left( - \frac{(-x)^n}{n} \right)
=
\sum_{i\in I} w_i \sum_{n=1}^{m-1} a_n^{(i)} x^n
=
\sum_{i\in I} \frac{w_i}{ 2 \pi \mathrm{i} } \oint_{ |z| = r } \frac{1-z^{m-1}}{1-z} h_i \left( \frac{x}{z} \right) dz. \qedhere
\]
\end{proof}

\subsection{Integral representation for iterates} \label{subsection: Integral representation for iterates}

We notice that the integrand in the integral representations in Proposition \ref{prop: integral representation} takes the following shape; integral kernel times $k$-th component of $v\big(f_i(x);0\big)$. That is, it ``preserves'' the shape of the vectors. Thus they extend naturally to iterates. For $i \in I$ and $z \in \partial \mathbb{D}_r$, define $f_{i, z}: \mathbb{D}_r \to \mathbb{D}_r$ by $f_{i, z}(x) = f_i \left( \frac{x}{z} \right)$ for $x \in \mathbb{D}_r$. Here $f_{i, z}(x) \in \mathbb{D}_r$ follows from $x/z \in \mathbb{D}_1$ and Lemma \ref{lemma: real contraction implies complex contraction}. For $\lbar{i} = (i_1, i_2, \ldots, i_n) \in I^n$ and $(z_1, \ldots, z_n) \in (\partial \mathbb{D}_r)^n$, define $\mathscr{F}_{z_n, \ldots, z_1}^{i_n, \ldots, i_1}: \mathbb{D}_r \to \mathbb{D}_r$ by
\begin{equation} \label{eq: definition of mathscr F}
\mathscr{F}_{z_n, \ldots, z_1}^{i_n, \ldots, i_1}(x) = f_{i_n, z_n} \circ \cdots \circ f_{i_1, z_1} (x).
\end{equation}

\begin{proposition} \label{prop: integral representation for iterates}
Let $x \in \mathbb{D}_r$ and $c \in \mathbb{C}$. Under the notations above, we have for any $k,n \geq 1$,
\begin{equation} \label{eq: integral representation of power of T}
\bigg( T^n v(x \, ; \, c) \bigg)_k
=
- \sum_{\lbar{i} \in I^n} \frac{w_{\lbar{i}}}{ \big( 2\pi\mathrm{i} \big)^n } \oint_{|z_1| = r} \cdots \oint_{|z_n| = r} \left( \prod_{\ell = 1}^n \frac{ 1 }{ 1 - z_\ell } \right) \frac{ \left( - \mathscr{F}_{z_n, \ldots, z_1}^{i_n, \ldots, i_1}(x) \right)^k }{ k} dz_n \cdots dz_1,
\end{equation}
and for any $n \geq 1$ and $1 \leq k \leq m-1$,
\begin{equation} \label{eq: integral representation of power of Tm}
\bigg( {T_m}^n v(x \, ; \, c) \bigg)_k
=
- \sum_{\lbar{i} \in I^n} \frac{w_{\lbar{i}}}{ \big( 2\pi\mathrm{i}\big)^n } \oint_{|z_1| = r} \cdots \oint_{|z_n| = r} \left( \prod_{\ell = 1}^n \frac{ 1 - {z_\ell}^{m-1} }{ 1 - z_\ell } \right) \frac{ \left( - \mathscr{F}_{z_n, \ldots, z_1}^{i_n, \ldots, i_1}(x) \right)^k }{ k} dz_n \cdots dz_1.
\end{equation}
Here, all contour integrals are taken over the positively oriented circle $|z_\ell|=r$.
\end{proposition}

\begin{proof}
First, equation \eqref{eq: integral representation of power of Tm} readily follows by induction on $n$ and Proposition \ref{prop: integral representation}, since we only exchange a finite sum with the integral. Let us prove equation \eqref{eq: integral representation of power of T}. By applying equation \eqref{eq: Tv} iteratively, we have
\[
T^n v(x \, ; \, c)
=
\sum_{\lbar{i}\in I^n} w_{\lbar{i}} \left\{ v \bigg( f_{\lbar{i}}(x) \, ; \, \ell_{\lbar{i}}(x) \bigg) - \, v \bigg( f_{\lbar{i}}(0) \, ; \, \ell_{\lbar{i}}(0) \bigg) \right\}.
\]
Here, for $\lbar i = (i_1, \ldots, i_n)$, we write $w_{\lbar i} = w_{i_1} \cdots w_{i_n}$, $f_{\lbar i} = f_{i_n}\circ \cdots \circ f_{i_1}$, and $\ell_{\lbar i} = \ell_{i_n}\circ f_{i_{n-1}}\circ \cdots \circ f_{i_1}$. Therefore, for every $k\geq 1$,
\begin{equation} \label{eq: the iterated Tnv}
\bigg( T^n v(x \, ; \, c) \bigg)_k
=
-\sum_{\lbar{i}\in I^n} w_{\lbar{i}} \left\{ \frac{\big( - f_{\lbar{i}}(x) \big)^k}{k} - \frac{\big( - f_{\lbar{i}}(0) \big)^k}{k} \right\}.
\end{equation}

Next, fix $k\geq 1$ and let
\[
h_k(w) \coloneqq - \frac{ (-w)^k}{k}.
\]
Then, by applying Lemma \ref{lemma: coboundary property} to $h = h_k\circ f_{i_n}$ and $w = \mathscr{F}_{z_{n-1}, \ldots, z_1}^{i_{n-1}, \ldots, i_1}(x)$,
\begin{align*}
&\sum_{\lbar{i} \in I^n} \frac{w_{\lbar{i}}}{ \big( 2\pi\mathrm{i} \big)^n } \oint_{|z_1| = r} \cdots \oint_{|z_n| = r} \left( \prod_{\ell = 1}^n \frac{ 1 }{ 1 - z_\ell } \right) h_k\circ f_{i_n} \left( \frac{ \mathscr{F}_{z_{n-1}, \ldots, z_1}^{i_{n-1}, \ldots, i_1}(x)}{z_n} \right) dz_n \cdots dz_1 \\
&=\sum_{\lbar{i} \in I^n} \frac{w_{\lbar{i}}}{ \big( 2\pi\mathrm{i} \big)^{n-1} } \oint_{|z_1| = r} \cdots \oint_{|z_{n-1}| = r} \\
& \hspace{50pt} \left( \prod_{\ell = 1}^{n-1} \frac{ 1 }{ 1 - z_\ell } \right) \left\{ h_k\circ f_{i_n} \left( \mathscr{F}_{z_{n-1}, \ldots, z_1}^{i_{n-1}, \ldots, i_1}(x) \right) - h_k\circ f_{i_n}(0) \right\} dz_{n-1}\cdots dz_1.
\end{align*}
Since the term $h_k\circ f_{i_n}(0)$ is holomorphic on $\mathbb D_r$, it vanishes after the next contour integration. We continue to apply Lemma \ref{lemma: coboundary property} iteratively, and we see that the right-hand side of equation \eqref{eq: integral representation of power of T} equals
\[
\sum_{\lbar i\in I^n} w_{\lbar i} \Big\{ h_k\circ f_{\lbar i}(x) - h_k\circ f_{\lbar i}(0) \Big\},
\]
which is exactly equation \eqref{eq: the iterated Tnv}. This completes the proof.
\end{proof}

\subsection{Integral operators and telescoping} \label{subsection: Integral operators and telescoping}

Using these integral representations, we now estimate $\bigg( \big( T^n - {T_m}^n \big) v(x\,;\,c)\bigg)_k$.

Denote by $\pazocal{H}(D)$ the set of holomorphic functions on a domain $D$. For $i \in I$ and $m \in \mathbb{N}$, we define $\pazocal{L}_i, \pazocal{R}_{i, m}, \pazocal{K}_{i, m}: \pazocal{H}(\mathbb{D}_r) \to \pazocal{H}(\mathbb{D}_r)$ by the following. For $h \in \pazocal{H}(\mathbb{D}_r)$,
\begin{gather*}
\pazocal{L}_i h(w) = \frac{1}{ 2 \pi \mathrm{i} } \oint_{ |z| = r } \frac{1}{1-z} \, \, h \left( f_i\left( \frac{w}{z} \right) \right) dz, \\[5pt]
\pazocal{R}_{i, m} h(w) = \frac{1}{ 2 \pi \mathrm{i} } \oint_{ |z| = r } \frac{z^{m-1}}{1-z} \, \, h \left( f_i\left( \frac{w}{z} \right) \right) dz, \\[5pt]
\pazocal{K}_{i,m}h = \pazocal{L}_i h - \pazocal{R}_{i,m}h.
\end{gather*}
These are well-defined. Indeed, for $w \in \mathbb{D}_r$ and $z \in \partial \mathbb{D}_r$ we have $w/z \in \mathbb{D}_1$, and $f_i(\mathbb{D}_1) \subset \mathbb{D}_r$ by Lemma \ref{lemma: real contraction implies complex contraction}. For each fixed $z \in \partial \mathbb{D}_r$, the map $w \longmapsto \frac{1}{1-z} \, \, h \circ f_i\left( \frac{w}{z} \right)$ is holomorphic. For any closed piecewise curve $\gamma$ inside a compact $K \subset \mathbb{D}_r$, by using the boundedness of $h$ on the compact set $\{ f_i(w/z) \, | \, w \in K, \, |z| = r\}$ and Fubini's theorem,
\[
\oint_\gamma \pazocal{L}_i h(w) dw = \frac{1}{ 2 \pi \mathrm{i} } \oint_{ |z| = r } \oint_\gamma \frac{1}{1-z} \, \, h \left( f_i\left( \frac{w}{z} \right) \right) dw dz = 0.
\]
Then Morera's theorem ensures that $\pazocal{L}_i h$ is holomorphic on $\mathbb{D}_r$. The same argument applies to $\pazocal{R}_{i,m}$, hence also to $\pazocal{K}_{i,m} = \pazocal{L}_i - \pazocal{R}_{i,m}$.

By Lemma \ref{lemma: coboundary property}, for any $h \in \pazocal{H}(\mathbb{D}_r)$
\begin{equation} \label{eq: operator L identity}
\pazocal{L}_i h(w) = h\big( f_i(w) \big) - h\big( f_i(0) \big).
\end{equation}
We define the oscillation seminorm $\| \cdot \|_{\mathrm{osc}}$ for $h \in \pazocal{H}(\mathbb{D}_r)$ by
\[
\| h \|_{\mathrm{osc}} = \sup_{u, v \in \mathbb{D}_r} \big| h(u) - h(v) \big|.
\]
Also, let 
\begin{equation} \label{eq: definition of Kr}
K(r) = \frac{1}{2 \pi} \oint_{|z| = r} \frac{|dz|}{|1-z|} \; \; \in \; (0, \infty).
\end{equation}

\begin{lemma} \label{lemma: basic properties of LRK}
Suppose $h \in \pazocal{H}( \mathbb{D}_r )$ satisfies $\| h \|_{\mathrm{osc}} < \infty$. Then for any $i \in I$ and $m\geq1$,
\begin{equation*}
\left\| \pazocal{L}_i h \right\|_{\mathrm{osc}} \leq \left\| h \right\|_{\mathrm{osc}}, \quad
\left\| \pazocal{R}_{i,m} h \right\|_{\mathrm{osc}} \leq r^{m-1} K(r) \left\| h \right\|_{\mathrm{osc}}, \quad
\left\| \pazocal{K}_{i,m} h \right\|_{\mathrm{osc}} \leq \big( 1 + r^{m-1} K(r) \big) \left\| h \right\|_{\mathrm{osc}}.
\end{equation*}
Also,
\begin{equation} \label{eq: LRKh0 is 0}
\pazocal{L}_i h(0) = \pazocal{R}_{i,m} h(0) = \pazocal{K}_{i,m} h(0) = 0.
\end{equation}
\end{lemma}

\begin{proof}
First inequality for $\pazocal{L}_i$ is immediate from equation \eqref{eq: operator L identity}. For any $u, v \in \mathbb{D}_r$,
\begin{align*}
\left| \pazocal{R}_{i,m} h(u) - \pazocal{R}_{i,m} h(v) \right|
&= \left| \frac{1}{ 2 \pi \mathrm{i} } \oint_{ |z| = r } \frac{z^{m-1}}{1-z} \Big( \, \, h \circ f_i\left( \frac{u}{z} \right) - \, \, h \circ f_i\left( \frac{v}{z} \right) \Big) dz \right| \\
&\leq \frac{ \left\| h \right\|_{\mathrm{osc}} }{2 \pi} \oint_{|z| = r} \frac{ r^{m-1} |dz|}{|1-z|}
= r^{m-1} K(r) \left\| h \right\|_{\mathrm{osc}}.
\end{align*}
The inequality for $\pazocal{K}_{i,m}$ follows from the triangle inequality. The equation \eqref{eq: LRKh0 is 0} is an immediate consequence of Cauchy's integral theorem.
\end{proof}

The following telescoping expansion is crucial.
\begin{lemma} \label{lemma: telescoping of L and K}
Fix $m \in \mathbb{N}$ and write $\pazocal{K}_i = \pazocal{K}_{i,m}$ and $\pazocal{R}_i = \pazocal{R}_{i,m}$ for $i \in I$. For any $(i_1, \ldots, i_n) \in I^n$,
\[
\pazocal{L}_{i_1} \cdots \pazocal{L}_{i_n} - \pazocal{K}_{i_1} \cdots \pazocal{K}_{i_n}
=
\sum_{\ell = 1}^n \pazocal{L}_{i_1} \cdots \pazocal{L}_{i_{\ell-1}} \pazocal{R}_{i_{\ell}} \pazocal{K}_{i_{\ell+1}} \cdots \pazocal{K}_{i_n}.
\]
\end{lemma}

\begin{proof}
We prove by induction on $n$. For $n = 1$, the right-hand side is (defined as) $ \pazocal{R}_{i_1} = \pazocal{L}_{i_1} - \pazocal{K}_{i_1}$. If the identity is true for words of length $n$,
\begin{align*}
\pazocal{L}_{i_1} \cdots \pazocal{L}_{i_{n+1}} - \pazocal{K}_{i_1} \cdots \pazocal{K}_{i_{n+1}}
&= \pazocal{L}_{i_1} \left( \pazocal{L}_{i_2} \cdots \pazocal{L}_{i_{n+1}} - \pazocal{K}_{i_2} \cdots \pazocal{K}_{i_{n+1}} \right) + \left( \pazocal{L}_{i_1} - \pazocal{K}_{i_1} \right) \pazocal{K}_{i_2} \cdots \pazocal{K}_{i_{n+1}}
\\
&= \pazocal{L}_{i_1} \left( \sum_{\ell = 1}^n \pazocal{L}_{i_2} \cdots \pazocal{L}_{i_{\ell}} \pazocal{R}_{i_{\ell+1}} \pazocal{K}_{i_{\ell+2}} \cdots \pazocal{K}_{i_{n+1}} \right) + \pazocal{R}_{i_1} \pazocal{K}_{i_2} \cdots \pazocal{K}_{i_{n+1}} \\
&= \left( \sum_{\ell = 1}^n \pazocal{L}_{i_1} \cdots \pazocal{L}_{i_{\ell}} \pazocal{R}_{i_{\ell+1}} \pazocal{K}_{i_{\ell+2}} \cdots \pazocal{K}_{i_{n+1}} \right) + \pazocal{R}_{i_1} \pazocal{K}_{i_2} \cdots \pazocal{K}_{i_{n+1}} \\
&= \sum_{\ell = 1}^{n+1} \pazocal{L}_{i_1} \cdots \pazocal{L}_{i_{\ell-1}} \pazocal{R}_{i_{\ell}} \pazocal{K}_{i_{\ell+1}} \cdots \pazocal{K}_{i_{n+1}}. \qedhere
\end{align*}
\end{proof}

We are now ready to prove that the error $\sum_{n=0}^{N-1} \Big( T^n v - T_m^n v \Big)_0$ is exponentially small in $m$.

\begin{proposition} \label{prop: evaluating the truncation}
Let $x \in \mathbb{D}_r$, and $c \in \mathbb{C}$. Let $\rho_m = r^{m-1} K(r)$. For any $N,m \in \mathbb{N}$ with $m \geq 2$,
\begin{align*}
\left| \sum_{n=0}^{N-1} \hspace{-3pt} \bigg( \hspace{-3pt} ( T^n - {T_m}^n ) v(x \, ; \, c) \hspace{-3pt}\bigg)_{\!\!0} \right|
\leq 2\! \left( \! \log \frac{1}{1-rC} \right) \hspace{-3pt} \left( \!\frac{ {(1+\rho_m)}^{\!N-1} -\!1 - (N-1)\rho_m}{\rho_m} \hspace{-3pt} \right) 
\!+ \! \frac{ 2 (N-1) r^m C^m }{ m \big( 1 - rC \big) },
\end{align*}
where $C = \max_{i \in I} | f_i^{\top}(0) |$.
\end{proposition}

\begin{proof}
Fix $x \in \mathbb{D}_r$ and $c \in \mathbb{C}$. For each $(i_1, \ldots, i_n) \in I^n$, $1 \leq k \leq m-1$, and $n \geq 1$, we define
\begin{equation*}
A_n^{(k)}(i_1, \ldots, i_n)(x)
=
- \frac{1}{(2\pi\mathrm{i})^n} \oint_{|z_1| = r} \cdots \oint_{|z_n| = r} \left( \prod_{\ell = 1}^n \frac{ 1 }{ 1 - z_\ell } \right) \frac{ \left( - \mathscr{F}_{z_n, \ldots, z_1}^{i_n, \ldots, i_1}(x) \right)^k }{k} dz_n \cdots dz_1,
\end{equation*}
\begin{equation*}
B_{m, n}^{(k)}(i_1, \ldots, i_n)(x)
=
- \frac{1}{ (2\pi\mathrm{i})^n } \oint_{|z_1| = r} \cdots \oint_{|z_n| = r} \left( \prod_{\ell = 1}^n \frac{ 1 - {z_\ell}^{m-1} }{ 1 - z_\ell } \right) \frac{ \left( - \mathscr{F}_{z_n, \ldots, z_1}^{i_n, \ldots, i_1}(x) \right)^k }{k} dz_n \cdots dz_1.
\end{equation*}
Define $h^{(k)} \in \pazocal{H}(\mathbb{D}_r)$ by $h^{(k)}(w) = - \frac{(-w)^k}{k}$ for each $k \geq 1$. Then, for any $x \in \mathbb{D}_r$,
\[
A_n^{(k)}(i_1, \ldots, i_n)(x) = \pazocal{L}_{i_1} \cdots \pazocal{L}_{i_n} h^{(k)}(x), \qquad
B_{m, n}^{(k)}(i_1, \ldots, i_n)(x) = \pazocal{K}_{i_1, m} \cdots \pazocal{K}_{i_n, m} \, h^{(k)}(x).
\]
By Lemma \ref{lemma: telescoping of L and K} we have
\begin{align} \label{eq: telescoping applied}
A_n^{(k)}(i_1, \ldots, i_n)(x) - B_{m, n}^{(k)}(i_1, \ldots, i_n)(x)
=
\sum_{\ell = 1}^n \pazocal{L}_{i_1} \cdots \pazocal{L}_{i_{\ell-1}} \pazocal{R}_{i_{\ell}} \pazocal{K}_{i_{\ell+1}, m} \cdots \pazocal{K}_{i_n, m} \, h^{(k)}(x).
\end{align}

Let $g \in \pazocal{H}(\mathbb{D}_r)$. For any composition $S$ of the operators $\pazocal{L}_i$, $\pazocal{R}_{i,m}$, $\pazocal{K}_{i,m}$, we have $Sg(0) = 0$ by Lemma \ref{lemma: basic properties of LRK}. Hence $| Sg(w) | = | Sg(w) - Sg(0) | \leq \left\| Sg \right\|_{\mathrm{osc}}$ for any $w \in \mathbb{D}_r$.

Thus, combining Proposition \ref{prop: integral representation for iterates}, equation \eqref{eq: telescoping applied}, Lemma \ref{lemma: basic properties of LRK}, and $\| h^{(k)} \|_{\mathrm{osc}} \leq \frac{2r^k}{k}$, we have for any $n \geq 1$ and $1\leq k\leq m-1$,
\begin{align}
\left| \bigg( \big( T^n - {T_m}^n \big) v(x \, ; \, c) \bigg)_k \right|
&\leq \sum_{ \lbar{i} \in I^n} w_{\lbar{i}}
\Big| A_n^{(k)}(i_1, \ldots, i_n)(x) - B_{m, n}^{(k)}(i_1, \ldots, i_n)(x) \Big| \nonumber \\
&\leq
\sum_{ \lbar{i} \in I^n} w_{\lbar{i}} \Big\| \sum_{\ell = 1}^n \pazocal{L}_{i_1} \cdots \pazocal{L}_{i_{\ell-1}} \pazocal{R}_{i_{\ell}, m} \pazocal{K}_{i_{\ell+1}, m} \cdots \pazocal{K}_{i_n, m} \, h^{(k)} \Big\|_{\mathrm{osc}} \nonumber \\
&\leq \sum_{\ell = 1}^n 1^{\ell-1} \rho_m \big( 1 + \rho_m \big)^{n-\ell} \| h^{(k)} \|_{\mathrm{osc}}
\leq \frac{2 r^k}{k} \Big( \big( 1 + \rho_m \big)^n - 1 \Big). \label{eq: final bound for k-th entry}
\end{align}

For the $0$-th entry, we have for any $n \geq 1$,
\begin{flalign}
&\bigg( \big( T^n - {T_m}^n \big) v(x \, ; \, c) \bigg)_0 & \nonumber \\
&= \sum_{i \in I} w_i
\left\{
\sum_{k=1}^{m-1} \left( f_i^{\top}(0) \right)^k
\bigg( \big( T^{n-1} - {T_m}^{n-1} \big) v(x \, ; \, c) \bigg)_k
+ \sum_{k=m}^\infty \left( f_i^{\top}(0) \right)^k
\bigg( T^{n-1} v(x \, ; \, c) \bigg)_k
\right\}. \label{eq: final bound for 0-th entry}
\end{flalign}
Also note that for any $n \geq 0$,
\begin{equation} \label{final bound for k-th entry of T only}
\left| \bigg( T^n v(x \, ; \, c) \bigg)_k \right| = \left| \sum_{\lbar{i} \in I^n} w_{\lbar{i}} \pazocal{L}_{i_1} \cdots \pazocal{L}_{i_n} h^{(k)}(x) \right| \leq \sum_{\lbar{i} \in I^n} w_{\lbar{i}} \| h^{(k)} \|_{\mathrm{osc}} \leq \frac{2r^k}{k}.
\end{equation}
Let $C = \max_i \left| f_i^{\top}(0) \right|$. Combining equations \eqref{eq: final bound for k-th entry}, \eqref{eq: final bound for 0-th entry}, and \eqref{final bound for k-th entry of T only},
\begin{align*}
& \left| \sum_{n=0}^{N-1} \bigg( ( T^n - {T_m}^n ) v(x \, ; \, c) \bigg)_0 \right| \\
&\leq \sum_{n=1}^{N-1} \sum_{k=1}^{m-1} C^k \left| \bigg( (T^{n-1} - {T_m}^{n-1}) v(x \, ; \, c) \bigg)_k \right|
+ \sum_{n=1}^{N-1} \sum_{k=m}^\infty C^k \left| \bigg( T^{n-1} v(x \, ; \, c) \bigg)_k \right| \\
&\leq \sum_{n=2}^{N-1} \sum_{k=1}^{m-1} \frac{2r^k C^k}{k}\Big( \big( 1 + \rho_m \big)^{n-1} - 1 \Big)
+ (N-1)\sum_{k=m}^\infty C^k \frac{2 r^k}{k} \\
&\leq 2 \left( \log \frac{1}{1-rC} \right) \left( \frac{ (1+\rho_m)^{N-1} -1 - (N-1)\rho_m}{\rho_m} \right) 
+ \frac{ 2 (N-1) r^m C^m }{ m \big( 1 - rC \big) }. & \qedhere
\end{align*}
\end{proof}

Recall $K(r)$, which we defined in equation \eqref{eq: definition of Kr} as
\[ K(r) = \frac{1}{2 \pi} \oint_{|z| = r} \frac{|dz|}{|1-z|}. \]
We have the following bound for $K(r)$, and its proof is included in Appendix \ref{appendix: proof of minor lemmas}
\begin{lemma} \label{lemma: bound for Kr}
We have
\[
K(r) \leq \min \left\{ \frac{r}{\sqrt{1-r^2}}, \hspace{8pt} \frac{2r}{\pi(1+r)} \left( \frac{\pi}{2} + \log \frac{1+r}{1-r} \right) \right\}.
\]
\end{lemma}

Combining Proposition \ref{prop: cutting the tail of the Neumann series}, Proposition \ref{prop: evaluating the truncation}, and Lemma \ref{lemma: bound for Kr}, we obtain the following.
\begin{theorem} \label{thm: precise bounds}
Under the notations of Theorem \ref{thm: main theorem 1}, take $0<r<1$ so that $f_i([-1,1]) \subset [-r,r]$ for all $i \in I$. Let
\[
E = \frac{1}{ 1 - r } \left( \sum_{i \in I} w_i \frac{ \left| f_i^{\top}(0) \right|}{ 1 + \sqrt{ 1 - {\left| f_i^{\top}(0) \right|}^2 } } \right) \left( \sum_{i \in I} w_i d_{\mathrm{hyp}}\big( f_i(0), 0 \big) \right).
\]
Let $C = \max_{i \in I} | f_i^{\top}(0) |$. Then $0 \leq C < 1$, and we have for any natural number $n, m \geq 2$ that
\begin{align*}
\left| \lambda - \sum_{j=0}^{n-1} \big( {T_m}^j \boldsymbol{v} \big)_0 \right|
\leq E r^{n-1} \hspace{-2pt}
+ 2 \left( \log \frac{1}{1-rC} \right) \hspace{-3pt} \left( \frac{ \big( 1 + \rho_m \big)^{n-1} -1 - (n-1)\rho_m}{\rho_m} \hspace{-2pt} \right) \hspace{-2pt}
+ \frac{ 2 (n-1) r^m C^m }{ m \big( 1 - rC \big) }.
\end{align*}
Here, $\rho_m = r^{m-1} K(r)$ and
\[
0 < K(r) \leq \min \left\{ \frac{r}{\sqrt{1-r^2}}, \hspace{8pt} \frac{2r}{\pi(1+r)} \left( \frac{\pi}{2} + \log \frac{1+r}{1-r} \right) \right\}.
\]
\end{theorem}

\begin{proof}
By equation \eqref{eq: bounding the transpose of fi}, we have $0 \leq C <1$. By triangle inequality,
\begin{flalign*}
\left| \lambda - \sum_{j=0}^{n-1} \big( {T_m}^j \boldsymbol{v} \big)_0 \right|
&= \left| \sum_{j=0}^\infty \big( T^j \boldsymbol{v} \big)_0 - \sum_{j=0}^{n-1} \big( {T_m}^j \boldsymbol{v} \big)_0 \right| &
\end{flalign*} \\[-25pt]
\begin{flalign*}
& \leq \left| \sum_{j=0}^\infty \big( T^j \boldsymbol{v} \big)_0 - \sum_{j=0}^{n-1} \big( T^j \boldsymbol{v} \big)_0 \right|
+ \left| \sum_{j=0}^{n-1} \big( T^j \boldsymbol{v} \big)_0 - \sum_{j=0}^{n-1} \big( {T_m}^j \boldsymbol{v} \big)_0 \right| & \\
& \leq \left| \sum_{j=0}^\infty \big( T^j \boldsymbol{v} \big)_0 - \sum_{j=0}^{n-1} \big( T^j \boldsymbol{v} \big)_0 \right|
+ \sum_{i \in I} w_i \left| \sum_{j=0}^{n-1} \big( T^j v\big( f_i(0)\, ;\, \ell_i(0) \big) \big)_0 - \sum_{j=0}^{n-1} \big( {T_m}^j v\big( f_i(0)\, ;\, \ell_i(0) \big) \big)_0 \right|. &
\end{flalign*}
The claim follows by applying Proposition \ref{prop: cutting the tail of the Neumann series} to the first term and Proposition \ref{prop: evaluating the truncation} to the second term, and using $\sum_{i \in I} w_i = 1$.
\end{proof}

\begin{proof}[Proof of Theorem \ref{thm: main theorem 2}]
Set
\[
\rho_M = r^{M-1}K(r), \qquad C=\max_{i\in I}|f_i^\top(0)|.
\]
By equation \eqref{eq: bounding the transpose of fi}, we have $0 \leq C <1$.

Now let $\vep>0$ be given. Since $0<r<1$, we can take a natural number $N=O(\log(1/\varepsilon))$ so that
\[
E r^{N-1} < \frac{\vep}{2}.
\]
Next, we consider $M$ satisfying $(N-1)\rho_M \leq 1$. Note that for any $t \geq 0$ we have $(1+t)^{N-1} \leq e^{(N-1)t}$ and $e^t-1-t \leq \frac12 t^2 e^t$. These combined, we get
\begin{align*}
\frac{ (1+\rho_M)^{N-1} -1 - (N-1)\rho_M}{\rho_M}
&\leq \frac{ e^{(N-1)\rho_M} - 1 - (N-1)\rho_M}{\rho_M} \\
&\leq \frac{ (N-1)^2 \rho_M^2 e^{(N-1)\rho_M} }{ 2\rho_M }
\leq \frac e2 (N-1)^2 \rho_M.
\end{align*}
By the definition of $\rho_M$ and $0<r<1$, we can take a natural number $M = O\big(\log(N/\varepsilon)\big)$ with
\[
(N-1)\rho_M \leq 1, \qquad e \times \left( \log \frac{1}{1-rC} \right) (N-1)^2 \rho_M
+ \frac{ 2 (N-1) r^M C^M }{ M \big( 1 - rC \big) } < \frac{\vep}{2}.
\]
With these choices of $N$ and $M$, by Theorem \ref{thm: precise bounds},
\[
\left| \lambda - \sum_{j=0}^{N-1} \big( {T_M}^j \, \boldsymbol{v} \big)_0 \right|
\leq Er^{N-1} + e \times \left( \log \frac{1}{1-rC} \right) (N-1)^2 \rho_M
+ \frac{ 2 (N-1) r^M C^M }{ M \big( 1 - rC \big) } < \vep.
\]

By the definition of $T$, its $(k,n)$-entry can be assembled in $O\!\big(\min\{k,n\} \times \#I \big)$ arithmetic operations. Therefore, the computational cost to calculate $T_M$ is
\[
O\!\left( \#I \sum_{k=1}^M \sum_{n=1}^M \min\{k,n\}\right)
= O\!\left(\frac{M(M+1)(2M+1)}{6}\right)
= O\!\left(M^3\right).
\]
Moreover, forming the first $M$ entries of vectors $T_M^k \boldsymbol{v}$ for all $0\leq k \leq N-1$ via repeated finite matrix--vector multiplication costs $O(NM^2)$ arithmetic operations. Hence the total complexity is
\[
O\!\left(M^3\right) + O\!\left(NM^2\right).
\]
In particular, when $N = O \! \left( \log{(1/\vep)} \right)$ and $M = O \! \left( \log{(N/\vep)} \right)$, this yields the polynomial-time bound
\[
O\!\left(M^3\right) + O\!\left(NM^2\right)
= O\!\left(\bigl(\log(1/\vep)\bigr)^3\right).
\]
\end{proof}

\begin{remark}
Consider parameterized families of invertible positive $2 \times 2$ matrices $\{A_i(t)\}_{i \in I}$ and probability vectors $\big( w_i(t) \big)_{i \in I}$, and denote by $\lambda(t)$ the Lyapunov exponent. The integral representations for $T$ (and $T_m$) in Proposition \ref{prop: integral representation} and Proposition \ref{prop: integral representation for iterates} can be used to prove analyticity of $\lambda(t)$ under uniform positivity of entries of $\{A_i(t)\}_{i \in I}$ and suitable regularity assumptions on $t$.
\end{remark}

\section{Proof of Theorem \ref{thm: main theorem 3} and Applications} \label{section: proof of main theorem 3 and applications}

\subsection{Proof of Theorem \ref{thm: main theorem 3}}

For $\alpha, \beta \in \widehat{\mathbb{R}}$ with $\alpha \ne \beta$, let $L(\alpha, \beta) \subset \widehat{\mathbb{R}}$ be the connected component of $\widehat{\mathbb{R}} \setminus \{\alpha,\beta\}$ whose closure is the oriented closed arc from $\alpha$ to $\beta$ in the standard cyclic order on $\widehat{\mathbb{R}}$.

The following ``pushing lemma'' is crucial for constructing an invariant arc in the proof of Theorem \ref{thm: main theorem 3}.
\begin{lemma} \label{lemma: Mobius pushing lemma}
Let $f \in \mathrm{PGL}_2(\mathbb{R})$ be a hyperbolic M\"obius transformation such that its attracting point $\alpha$ is a finite real number. Let $\beta \in \widehat{\mathbb{R}}$ be its repelling fixed point. Let $p = f^{-1}(\widehat{\infty}) \in \widehat{\mathbb{R}}$ be its pole. Then, for any $x \in \mathbb{R}$,
\begin{enumerate}
\item If $x \in L(\beta, \alpha) \cap L(p, \alpha) \cap (-\infty, \alpha)$, then $x < f(x)$, and
\item if $x \in L(\alpha, \beta) \cap L(\alpha, p) \cap (\alpha, \infty)$, then $f(x) < x$.
\end{enumerate}
\end{lemma}

\begin{proof}
First assume $p \ne \widehat{\infty}$. Then, $\widehat{\infty}$ is not a fixed point, so $\alpha, \beta \in \mathbb{R}$. Define $g(y) = f(y) - y$ on $\mathbb{R} \setminus \{p\}$, then $g$ is continuous, and it is $0$ if and only if $y \in \{\alpha, \beta\}$. Let $I_1 = L(\beta, \alpha) \cap L(p, \alpha) \cap (-\infty, \alpha)$. Since $I_1 \cap \{ \alpha, \beta, p \} = \varnothing$ and $I_1$ is connected, $g$ has a constant sign on $I_1$. Now,
\[
\lim_{y \to \alpha-0} \frac{ f(y) - f(\alpha) }{ y - \alpha } = f'(\alpha) < 1.
\]
Then, for any $y < \alpha$ close enough to $\alpha$, we have $g(y) = f(y) - y > f(\alpha) - \alpha = 0$. Thus, $g(x) > 0$ for all $x \in I_1$. The proof of assertion (2) is identical.

Next, assume $p = \widehat{\infty}$. Then, $\widehat{\infty}$ is a fixed point, and $\beta = \widehat{\infty}$. Also, we have $f(x) = ax+b$ with some $a, b \in \mathbb{R}$. Since $|f'(\alpha)| < 1$, we have $a \leq |a| < 1$. Therefore, $g(y) = f(y) - y = (a-1)(y-\alpha)$ is always positive for $y < \alpha$ and negative for $y > \alpha$, proving the claim.
\end{proof}

\begin{lemma} \label{lemma: existence of conjugation and invariant arc}
Let $\{f_i\}_{i \in I} \subset \mathrm{PGL}_2(\mathbb{R})$ be a family of M\"obius transformations. For a M\"obius transformation $\phi \in \mathrm{PGL}_2(\mathbb{R})$ and $i \in I$, define $g_i \in \mathrm{PGL}_2(\mathbb{R})$ by
\[
g_i = \phi \circ {f_i} \circ \phi^{-1}.
\]
Let $J$ be a closed arc, and suppose $\phi(J) = [-1,1]$. Then, $g_i( [-1, 1] ) \subset (-1, 1)$ for all $i \in I$ if and only if $J$ is an invariant arc for $\{f_i\}_{i \in I}$.
\end{lemma}

\begin{proof}
Assume that $\phi(J) = [-1,1]$ and $g_i( [-1, 1] ) \subset (-1, 1)$ for all $i \in I$. Since $\phi$ is a homeomorphism of $\widehat{\mathbb{R}}$, we have for any $i \in I$
\[
f_i(J) = \phi^{-1} \circ g_i \circ \phi(J) = \phi^{-1} \circ g_i ( [-1, 1] ) \subset \phi^{-1}(-1,1) = \mathring{J}.
\]
Conversely, if $J$ is an invariant arc for $\{f_i\}_{i \in I}$,
\[
g_i([-1,1]) = \phi \circ {f_i} \circ \phi^{-1} ([-1,1]) = \phi \circ f_i( J ) \subset \phi ( \mathring{J} ) = (-1,1).
\]
\end{proof}

\begin{lemma} \label{lemma: existence of invariant arc cand coordinate change}
Let $\{A_i\}_{i\in I}$ be a finite family of invertible non-negative $2 \times 2$ matrices. Then, the following are equivalent.
\begin{enumerate}
\item $\{[F(A_i)]\}_{i\in I}$ admits an invariant arc $J$.
\item There is $P \in \mathrm{GL}_2(\mathbb{R})$ such that $M_i = P A_i P^{-1}$ is a positive matrix for every $i \in I$.
\end{enumerate}
\end{lemma}

\begin{proof}
$\big[$ (1) $\Rightarrow$ (2) $\big]$ Let $f_i = [F(A_i)]$ for each $i \in I$. Suppose $J$ is an invariant arc for $\{f_i\}_{i \in I}$. Take $\phi \in \mathrm{PGL}_2(\mathbb{R})$ so that $\phi(J) = [-1,1]$. Take any representation $Q \in \mathrm{GL}_2(\mathbb{R})$ of $\phi$ and let $G_i = Q F(A_i) Q^{-1}$ for each $i \in I$. Let $g_i = [G_i] \in \mathrm{PGL}_2(\mathbb{R})$, then $g_i = \phi \circ f_i \circ \phi^{-1}$, so by Lemma \ref{lemma: existence of conjugation and invariant arc} we have
\[
g_i([-1,1]) \subset (-1,1)
\]
for every $i \in I$.

Now, by the definition of $F$ in equation \eqref{eq: definition of F}, we have for any $M \in \mathrm{GL}_2(\mathbb{R})$,
\begin{equation} \label{eq: another representation of F}
F(M) = H M H^{-1} \hspace{7pt} \text{with} \hspace{7pt} H=
\begin{pmatrix}
1 & -1 \\
1& 1
\end{pmatrix}.
\end{equation}

Fix $i \in I$ and let $G_i = \begin{pmatrix} a & b \\ c & d \end{pmatrix}$. We claim that $cx + d > 0$ for all $x \in [-1,1]$. Assume otherwise. Then we have $cx_0 + d < 0$ for some $x_0 \in [-1,1]$. Since $g_i$ cannot have a pole in $[-1,1]$, we have $cx + d < 0$ for all $x \in [-1,1]$. By $g_i([-1,1]) \subset (-1,1)$, we have $|g_i(\pm 1)| < 1$. Thus,
\[
|a+b| < -(c+d) =:U, \quad |-a+b| < -(-c+d)=:V.
\]
Then, $(a+b)-(b-a) \leq |a+b| + |b-a| < U+V$, and
\begin{align*}
\mathrm{Tr}(G_i)
=
a+d
&= \frac{1}{2} \Big( (a+b)-(b-a) \Big) + \frac{1}{2} \Big( (c+d)+(-c+d) \Big) \\
&< \frac{1}{2}(U+V) - \frac{1}{2}(U+V) = 0.
\end{align*}
However, by equation \eqref{eq: another representation of F}, we have $\mathrm{Tr}(G_i) = \mathrm{Tr}(F(A_i)) = \mathrm{Tr}(A_i) \geq 0$, a contradiction. Thus, $cx+d > 0$ for all $x \in [-1,1]$.

The inverse map of $F$ satisfies $F^{-1}(N) = H^{-1} N H$. We define $M_i = F^{-1}(G_i)$ for $i \in I$. Then,
\[
M_i = F^{-1}(G_i) =
\frac{1}{2} \begin{pmatrix}
(a+b) + (c+d) & (b-a)+(d-c) \\
(c+d) - (a+b)& (d-c)-(b-a)
\end{pmatrix}.
\]
Since we have $g_i(\pm1) \in (-1,1)$ and $c + d > 0$ and $-c+d > 0$,
\[
|a+b| < c+d, \quad |b-a| < d-c.
\]
Applying these inequalities, we see that all the entries of $M_i$ are positive.

The following calculation completes the proof, where $P = H^{-1}QH$.
\[
M_i = F^{-1}(G_i) = H^{-1} Q F(A_i) Q^{-1} H = \bigg( H^{-1} Q H \bigg) A_i \bigg( H^{-1} Q H \bigg)^{-1}.
\]

$\big[$ (2) $\Rightarrow$ (1) $\big]$ Suppose $M_i = P A_i P^{-1}$ is a positive matrix for some $P \in \mathrm{GL}_2(\mathbb{R})$ for all $i \in I$. Then,
\begin{equation*}
F(M_i)
= H P A_i P^{-1} H^{-1}
= H P H^{-1} H A_i H^{-1} H P^{-1} H^{-1}
= \bigg( HPH^{-1} \bigg) F(A_i) {\bigg( HPH^{-1} \bigg)}^{-1}.
\end{equation*}
Let $\phi = \big[ HPH^{-1} \big] \in \mathrm{PGL}_2(\mathbb{R})$. Then the above equation implies
\[
[F(M_i)] = \phi \circ f_i \circ \phi^{-1}.
\]
Let $g_i = [F(M_i)]$. Since $\{M_i\}_{i \in I}$ are positive matrices, $g_i([-1,1]) \subset (-1,1)$. Let $J = \phi^{-1}([-1,1])$. Then by Lemma \ref{lemma: existence of conjugation and invariant arc}, $J$ is an invariant arc for $\{f_i\}_{i \in I}$.
\end{proof}

\begin{proposition} \label{prop: heteroclinic connection and invariant arc}
Let $\{A_i\}_{i\in I}$ be a finite family of invertible non-negative $2 \times 2$ matrices. Then, the following are equivalent.
\begin{enumerate}
\item $\{A_i\}_{i\in I}$ has no generalized heteroclinic connections of depth $\infty$.
\item $\{A_i\}_{i\in I}$ has no generalized heteroclinic connections of depth $2$.
\item $\{[F(A_i)]\}_{i\in I}$ admits an invariant arc.
\end{enumerate}
\end{proposition}

\begin{proof}
$\big[$ (3) $\Rightarrow$ (1) $\big]$ Let $f_i = [F(A_i)]$ for each $i \in I$, and assume $J$ is an invariant arc for $\{f_i\}_{i\in I}$. Take $\phi \in \mathrm{PGL}_2(\mathbb{R})$ such that $\phi(J) = [-1,1]$ and let
\[
g_f = \phi \circ f \circ \phi^{-1}
\]
for $f \in \pazocal{F}_\infty$. Then we have $g_f([-1,1]) \subset (-1,1)$ for every $f \in \pazocal{F}_1$ by Lemma \ref{lemma: existence of conjugation and invariant arc}, and thus $g_f([-1,1]) \subset (-1,1)$ for every $f \in \pazocal{F}_\infty$. By Lemma \ref{lemma: small image implies contraction}, $g_f$ is a strict contraction with respect to the hyperbolic metric. Also, $((-1,1), d_{\mathrm{hyp}})$ is complete. Then, by Banach's fixed point theorem, there is a unique attracting fixed point $\alpha \in (-1,1)$ of $g_f$, satisfying $|g_f'(\alpha)| < 1$. Hence $g_f$ is hyperbolic for every $f \in \pazocal{F}_\infty$.

Since conjugacy preserves asymptotic behavior around fixed points, every $f \in \pazocal{F}_\infty$ is hyperbolic, and we conclude that $\mathrm{Attr}_\infty \subset \mathring{J}$, and $\mathrm{Rep}_\infty \subset \widehat{\mathbb{R}} \setminus \mathring{J}$ $\big($by the uniqueness of fixed point on $\mathring{J}$$\big)$. In particular, for any $f \in \{ \mathrm{id} \} \cup \pazocal{F}_\infty$, we have $f(\mathrm{Attr}_\infty) \cap \mathrm{Rep}_\infty \subset \mathring{J} \cap \big( \widehat{\mathbb{R}} \setminus \mathring{J} \big) = \varnothing$, proving the absence of generalized heteroclinic connections of depth $\infty$.

$\big[$ (1) $\Rightarrow$ (2) $\big]$ This is immediate from the definition, since depth $2$ implies depth $\infty$.

$\big[$ (2) $\Rightarrow$ (3) $\big]$ First, take a non-negative invertible matrix $A = \begin{pmatrix} p & q \\ r & s \end{pmatrix}$ and consider the fixed point equation $[F(A)](x) = x$;
\[
(p-q+r-s)x^2 + 2(q+r)x-(p+q-r-s) = 0.
\]
Then, the discriminant of the above equation is $4(p-s)^2 + 16qr \geq 0$ by $p,q,r,s \geq 0$. Hence every $f \in \pazocal{F}_2$ has a real fixed point, so no element of $\pazocal{F}_2$ is elliptic. (Note that we used $F(AB) = F(A) F(B)$ to ensure that any $f \in \pazocal{F}_2$ satisfies $f = [F(M)]$ for some non-negative matrix $M$.) If some $f \in \pazocal{F}_2$ is the identity, parabolic or an involution, we have $|f'(p)| = 1$ for some fixed point $p$. This implies $\mathrm{Attr}_2 \cap \mathrm{Rep}_2 \ne \varnothing$, contradicting the absence of generalized heteroclinic connections. Therefore, every $f \in \pazocal{F}_2$ is hyperbolic.

\begin{claim} \label{claim: attractors and repellers location}
We have $\mathrm{Attr}_2 \subset [-1,1]$ and $\mathrm{Rep}_2 \cap (-1,1) = \varnothing$.
\end{claim}
\begin{proof}
Since $f_i([-1,1]) \subset [-1,1]$ for each $i \in I$, we have $f([-1,1]) \subset [-1,1]$ for every $f \in \pazocal{F}_2$. If $f(\alpha) = \alpha \in \widehat{\mathbb{R}} \setminus [-1,1]$ and $|f'(\alpha)| < 1$ for some $\alpha \in \widehat{\mathbb{R}}$ and $f \in \pazocal{F}_2$, we have $f^n(x) \to \alpha \in \widehat{\mathbb{R}} \setminus [-1,1]$ as $n \to \infty$ for every $x \in (-1,1)$ that is not a fixed point. This contradicts $f^n([-1,1]) \subset [-1,1]$ for every $n \geq 1$. Thus, $\mathrm{Attr}_2 \subset [-1,1]$.
If $f(\beta) = \beta \in (-1,1)$ and $|f'(\beta)| > 1$ for some $f \in \pazocal{F}_2$, we have $f^{-n}(\widehat{\infty}) \to \beta \in (-1,1)$ as $n \to \infty$. Then, for large enough $n$, we have $f^{-n}(\widehat{\infty}) \in (-1,1)$, implying $f^n(x) = \widehat{\infty}$ for some $x \in (-1,1)$, which is a contradiction.
\end{proof}

Now, define $\mathrm{Pole}_2$ as the set of poles of M\"obius maps in $\pazocal{F}_2$.
\[
\mathrm{Pole}_2 = \left\{ p \in \widehat{\mathbb{R}} \, \setcond \, \text{$f(p) = \widehat{\infty}$ for some $f \in \pazocal{F}_2$.} \right\}.
\]
Note that $\mathrm{Pole}_2 \cap [-1,1] = \varnothing$. Also, $\mathrm{Attr}_2$ and $\mathrm{Rep}_2$ are finite sets. Let
\[
\beta_1 = \max \bigg\{ -2, \,\sup\big( \mathrm{Pole}_2 \cap (-\infty, -1) \big), \, \sup \big( \mathrm{Rep}_2 \cap (-\infty, -1] \big) \bigg\}\in [-2, -1],
\]
\[
\alpha_1 = \min \big( \mathrm{Attr}_2 \big) \in [-1,1], \qquad \alpha_2 = \max \big( \mathrm{Attr}_2 \big) \in [-1,1].
\]
\[
\beta_2 = \min \bigg\{ 2, \, \inf \big( \mathrm{Pole}_2 \cap (1, \infty) \big), \, \inf \big( \mathrm{Rep}_2 \cap [1, \infty) \big) \bigg\} \in [1, 2],
\]
where, by convention we define $\sup \varnothing = - \infty$ and $\inf \varnothing = + \infty$. Then, since $\mathrm{Attr}_2 \cap \mathrm{Rep}_2 = \varnothing$, we have $\beta_1 < \alpha_1 \leq \alpha_2 < \beta_2$.

\begin{claim} \label{claim: uniform pushing on core interval}
For every $f \in \pazocal{F}_2$, the following hold:
\[
\text{(1) If $x \in (\beta_1, \alpha_1)$, then $x < f(x)$.} \qquad
\text{(2) If $x \in (\alpha_2, \beta_2)$, then $f(x) < x$.}
\]
\end{claim}

\begin{proof}
Fix $f \in \pazocal{F}_2$. By the previous argument, $f$ is hyperbolic. Let $\alpha$ be its attracting fixed point, $\beta$ its repelling fixed point, and $p=f^{-1}(\widehat\infty)$ its pole. By Claim \ref{claim: attractors and repellers location}, we have $\alpha \in [-1,1]$ and $\beta \notin (-1,1)$. Also, $p\notin[-1,1]$ since $f([-1,1])\subset[-1,1]$.

\smallskip

Let $x\in(\beta_1,\alpha_1)$. Since $\alpha_1\le \alpha$, we have $x<\alpha$. By the definition of the arc $L(\beta, \alpha)$ and $\beta_1$, we have $x \in L(\beta, \alpha)$ regardless of how $\alpha$ and $\beta$ are aligned. Similarly $x \in L(p, \alpha)$, and we conclude
\[
x \in L(\beta,\alpha)\cap L(p,\alpha)\cap(-\infty,\alpha),
\]
so Lemma \ref{lemma: Mobius pushing lemma} gives $x<f(x)$. The proof for (2) is identical.
\end{proof}

We claim that
\begin{equation} \label{eq: boundary behavior in the proof of heteroclinic connection and invariant arc}
\beta_1 < f(\alpha_j) < \beta_2 \hspace{7pt} \text{for $j = 1,2$ and $f \in \pazocal{F}_2$.}
\end{equation}
Indeed, take any $f \in \pazocal{F}_2$. Then we have $\beta_1 \leq -1 \leq f(\alpha_j)$ for $j = 1,2$. If $\beta_1 = f(\alpha_j) $, we have $\beta_1 = -1$, implying $\beta_1 \in \mathrm{Rep}_2$ and existence of generalized heteroclinic connections of depth $2$. Thus, $\beta_1 < f(\alpha_j)$. Similarly, $f(\alpha_j) < \beta_2$.

Let $A' = \min \big\{ \alpha_1, \min_{j \in I} f_j(\alpha_2) \big\} (> \beta_1)$. If $A' = \alpha_1$, we have $f_i(A') < \beta_2$ for all $i \in I$ by equation \eqref{eq: boundary behavior in the proof of heteroclinic connection and invariant arc}. If $A' = f_j(\alpha_2)$ with some $j \in I$, then for any $i \in I$, we have $f_i(A') = f_i \circ f_j (\alpha_2) < \beta_2$, again by equation \eqref{eq: boundary behavior in the proof of heteroclinic connection and invariant arc}.

Thus $f_i(A') < \beta_2$ for all $i \in I$. By the continuity of $f_i$ for all $i \in I$, there is $A \in (\beta_1, A')$ such that $f_i(A) < \beta_2$ for all $i \in I$.

Next, let $B' = \max \big\{ \alpha_2, \max_{j \in I} f_j(A) \big\} (< \beta_2)$. Then we claim that
\[
A < f_i(B') \hspace{7pt} \text{for all $i \in I$.}
\]
Indeed, if $B' = \alpha_2$, then for any $i \in I$ we have $f_i(B') = f_i(\alpha_2) \geq A' > A$. Suppose $B' = f_j(A)$ with some $j \in I$. If furthermore $f_i(B') \leq A$ for some $i \in I$, we have $f_i \circ f_j(A) \leq A$. However, since $\beta_1 < A < \alpha_1$, this contradicts the ``pushing'' Claim \ref{claim: uniform pushing on core interval}. Hence $A < f_i(B')$ for all $i \in I$.

Thus, by the continuity of $f_i$ for all $i \in I$, there is $B \in (B', \beta_2)$ such that $A < f_i(B)$ for every $i \in I$. Since $\beta_1 < A < \alpha_1$ we have $A < f_i(A)$ for all $i \in I$ by Claim \ref{claim: uniform pushing on core interval}, and by construction $f_i(A) \leq B' < B$ for every $i \in I$. Also, since $\alpha_2 < B < \beta_2$, we have $f_i(B) < B$ for all $i \in I$ by Claim \ref{claim: uniform pushing on core interval}.

We conclude that $f_i(A), f_i(B) \in (A, B)$ for all $i \in I$. By the monotonicity of $f_i$ on $[A,B]$ and the fact that their poles do not exist within $(\beta_1, \beta_2) \supset [A,B]$, we have $f_i([A,B]) \subset (A,B)$. Therefore, $[A,B]$ is an invariant arc.
\end{proof}

\begin{proof}[Proof of Theorem \ref{thm: main theorem 3}]
The equivalence of the stated conditions follows from Lemma \ref{lemma: existence of invariant arc cand coordinate change} and Proposition \ref{prop: heteroclinic connection and invariant arc}.
\end{proof}

\begin{example} \label{example: matrices with generalized heteroclinic connections}
Let us see some examples with generalized heteroclinic connections of depth $2$.
\begin{enumerate}
\item Let $I = \{1\}$ with weight $w_1 = 1$. Consider $A_1 = \begin{pmatrix} 2 & 1 \\ 0 & 2 \end{pmatrix}$. Then $[F(A_1)](x) = \frac{3x+1}{-x+5}$, which has $1$ as a double fixed point; it is parabolic, and $\{A_1\}$ has a generalized heteroclinic connection of depth $2$. Taking $\phi(x) = \frac{1}{1-x}$, we have $g(y) := \phi \circ [F(A_1)] \circ \phi^{-1}(y) = y + \frac{1}{4}$. Thus, by the partial sum formula in equation \eqref{eq: finite sum of T},
\[
\sum_{k = 0}^{n-1} \big( T^k \boldsymbol{v} \big)_0 = \log \left( \frac{5 - \frac{n-1}{n+3}}{2} \right) = \log 2 + \log \left( \frac{n+4}{n+3} \right) \xrightarrow[]{n \to \infty} \log 2.
\]
By the identity $A_1^n = 2^{n-1} \begin{pmatrix} 2 & n \\ 0 & 2 \end{pmatrix}$ for $n \geq 1$, we conclude that $\lambda = \sum_{k = 0}^\infty \big( T^k \boldsymbol{v} \big)_0$. $\big($The convergence is slow compared to the case for positive matrices.$\big)$
\item Let $I = \{1\}$, $w_1 = 1$, and $A_1 = \begin{pmatrix} 0 & 1+u \\ 1-u & 0 \end{pmatrix}$ with $u \in (0,1)$. Let $f_1 = [F(A_1)]$. Then $f_1$ is an involution. Also $f_1(x) = \frac{-x+u}{-ux+1}$, and we see that $f_1^n(0) = 0$ for even $n$ and $f_1^n(0) = u$ for odd $n$. Thus, the partial sum
\[
\sum_{k = 0}^{n-1} \big( T^k \boldsymbol{v} \big)_0 = \log \big( 1 - u f_1^{n-1}(0) \big)
\]
exhibits period-$2$ oscillation and does not converge. The Lyapunov exponent in this case is $\frac{1}{2} \log ( 1-u^2)$, which is the arithmetic mean of the orbit of $\sum_{k = 0}^{n-1} \big( T^k \boldsymbol{v} \big)_0$.
\end{enumerate}
\end{example}

\subsection{Proof of Applications}

Let us prove Proposition \ref{prop: toothless twins}. We start by identifying the degenerate cases.

\begin{lemma} \label{lemma: extreme digits and degeneracy}
Let $\mu$ be Lebesgue measure. The pair $(D_1, D_2)$ is degenerate if one of $D_1$, $D_2$ misses one of the extreme digits $0, b-1$.
\end{lemma}

\begin{proof}
If $0 \notin D_1$, then $| (D_1 + b-1) \cap D_2 | = | (D_1 + b) \cap D_2 | = 0$. Thus, for some $n,m \in \mathbb{N}_0$,
\begin{equation*}
A_{b-1} = 
\begin{pmatrix}
0 & n \\
0 & m
\end{pmatrix},
\end{equation*}
which implies that $(D_1, D_2)$ is degenerate. If $b-1 \notin D_1$, then $| D_1 \cap (D_2 + b) | = | (D_1 + 1) \cap (D_2 + b) | = 0$. Then again, we have for some $n, m \in \mathbb{N}_0$,
\begin{equation*}
A_0 = 
\begin{pmatrix}
n & 0 \\
m & 0
\end{pmatrix}.
\end{equation*}
Thus $(D_1, D_2)$ is degenerate. Similar arguments work for when $D_2$ misses these digits.

\end{proof}

\begin{lemma} \label{lemma: degeneracy of toothless twins}
Let $\mu$ be Lebesgue measure. Let $\tau, u \in \{0, 1, \ldots, b-1\}$, and
\begin{align*}
& D_1 = \{0, 1, \ldots, b-1\} \backslash \{ \tau \}, \\
& D_2 = \{0, 1, \ldots, b-1\} \backslash \{u\}.
\end{align*}
Then, the pair $(D_1, D_2)$ is degenerate if one of $\{ \tau, u\}$ is an extreme digit $\{0, b-1\}$ or $\tau + u \ne b-1$.
\end{lemma}

\begin{proof}
By Lemma \ref{lemma: extreme digits and degeneracy}, if $\tau$ or $u \in \{0, b-1\}$, the pair $(D_1, D_2)$ is degenerate. Suppose $\tau, u \notin \{0, b-1\}$ and $\tau + u \ne b-1$. We will see that $A_u$ is always degenerate.

Suppose either (i) $\tau < u$ and $b - 1 - u < \tau$, 
(ii) $u < \tau$ and $b - 1 - \tau < u$, or 
(iii) $\tau = u$ and $b - 1 < 2u$. 
In any of the above case, it can be checked that
\begin{align*}
&| (D_1 + u) \cap D_2 | = b-u-1, \quad | (D_1 + u) \cap (D_2 + b) | = u-1 \\
&| (D_1 + u+1) \cap D_2 | = b-u-1, \quad | (D_1 + u+1) \cap (D_2 + b) | = u-1.
\end{align*}
Thus,
\begin{equation*}
A_u = 
\begin{pmatrix}
b-u-1 & u-1 \\
b-u-1 & u-1
\end{pmatrix}.
\end{equation*}
Similarly, when either (iv) $\tau < u$ and $\tau < b - 1 - u$, 
(v) $u < \tau$ and $u < b - 1 - \tau$, or 
(vi) $\tau = u$ and $2u < b - 1$,
we have
\begin{equation*}
A_u = 
\begin{pmatrix}
b-u-2 & u \\
b-u-2 & u
\end{pmatrix}.
\end{equation*}
Therefore, the pair $(D_1, D_2)$ is degenerate under this assumption.
\end{proof}

\begin{proof}[Proof of Proposition \ref{prop: toothless twins}]
By Lemma \ref{lemma: degeneracy of toothless twins}, assume $u = b - 1 - \tau$ and $1 \leq \tau \leq b-2$. Suppose $\frac{b-1}{2} < \tau \leq b-2$. (The proof $1 \leq \tau \leq \frac{b-1}{2}$ is proved similarly.) Then,
\begin{equation*}
\# \big( \left( D_1 + s \right) \cap D_2 \big) = \left\{
\begin{array}{ll}
b-s-2
& \text{\quad for \, $0 \le s < b-\tau$,} \\
b-s
& \text{\quad for \, $b-\tau \le s \le b$.}
\end{array}
\right.
\end{equation*}

\begin{equation*}
\# \big( \left( D_1 + s \right) \cap \left(D_2 + b \right) \big) = \left\{
\begin{array}{ll}
s
& \text{\quad for \, $0 \le s < b-\tau$,} \\
s-2
& \text{\quad for \, $b-\tau \le s < 2b-2\tau-1$,} \\
s-1
& \text{\quad for \, $s = 2b-2\tau-1$,} \\
s-2
& \text{\quad for \, $2b-2\tau-1 < s \le b$.}
\end{array}
\right.
\end{equation*}
Therefore, the matrices $A_i$ are calculated as follows.
\begin{equation*}
A_i =
\begin{cases}
\vcenter{\hbox{$\begin{pmatrix}
\colA{b-i-2} & \colB{i}\\
\colA{b-i-3} & \colB{i+1}
\end{pmatrix}$}}
& \text{for \,} 0 \le i < b-\tau-1,\\[16pt]

\vcenter{\hbox{$\begin{pmatrix}
\colA{\tau-1} & \colB{b-\tau-1}\\
\colA{\tau} & \colB{b-\tau-2}
\end{pmatrix}$}}
& \text{for \,} i = b-\tau-1,\\[16pt]

\vcenter{\hbox{$\begin{pmatrix}
\colA{b-i} & \colB{i-2}\\
\colA{b-i-1} & \colB{i-1}
\end{pmatrix}$}}
& \text{for \,} b-\tau \le i < 2b-2\tau-2,\\[16pt]

\vcenter{\hbox{$\begin{pmatrix}
\colA{2\tau-b+2} & \colB{2b-2\tau-4}\\
\colA{2\tau-b+1} & \colB{2b-2\tau-2}
\end{pmatrix}$}}
& \text{for \,} i = 2b-2\tau-2,\\[16pt]

\vcenter{\hbox{$\begin{pmatrix}
\colA{2\tau-b+1} & \colB{2b-2\tau-2}\\
\colA{2\tau-b} & \colB{2b-2\tau-2}
\end{pmatrix}$}}
& \text{for \,} i = 2b-2\tau-1,\\[16pt]

\vcenter{\hbox{$\begin{pmatrix}
\colA{b-i} & \colB{i-2}\\
\colA{b-i-1} & \colB{i-1}
\end{pmatrix}$}}
& \text{for \,} 2b-2\tau-1 < i \le b-1.
\end{cases}
\end{equation*}

Now, the proof is easier if we take the transpose. $\big($By equivalence of (1) and (3) in Theorem \ref{thm: main theorem 3}, the existence of generalized heteroclinic connection of depth $2$ is equivalent to its existence for the transposed family.$\big)$ Let $G_i = 2F(A_i^{\top})$. Then, using the definition of $F$ in equation \eqref{eq: definition of F},
\begin{equation*}
G_i =
\begin{cases}
\vcenter{\hbox{$\begin{pmatrix}
\colA{2} & \colB{2b-4i-6 \hspace{4pt}}\\
\colA{0} & \colB{2b-4 \hspace{4pt}}
\end{pmatrix}$}}
& \text{for \,} 0 \le i < b-\tau-1,\\[16pt]

\vcenter{\hbox{$\begin{pmatrix}
\colA{-2} & \colB{-2b+4\tau+2 \hspace{4pt}}\\
\colA{0} & \colB{2b-4 \hspace{4pt}}
\end{pmatrix}$}}
& \text{for \,} i = b-\tau-1,\\[16pt]

\vcenter{\hbox{$\begin{pmatrix}
\colA{2} & \colB{2b-4i+2 \hspace{4pt}}\\
\colA{0} & \colB{2b-4 \hspace{4pt}}
\end{pmatrix}$}}
& \text{for \,} b-\tau \le i < 2b-2\tau-2,\\[16pt]

\vcenter{\hbox{$\begin{pmatrix}
\colA{3} & \colB{-6b+8\tau+9 \hspace{4pt}}\\
\colA{-1} & \colB{2b-3 \hspace{4pt}}
\end{pmatrix}$}}
& \text{for \,} i = 2b-2\tau-2,\\[16pt]

\vcenter{\hbox{$\begin{pmatrix}
\colA{1} & \colB{-6b+8\tau+5 \hspace{4pt}}\\
\colA{1} & \colB{2b-3 \hspace{4pt}}
\end{pmatrix}$}}
& \text{for \,} i = 2b-2\tau-1,\\[16pt]

\vcenter{\hbox{$\begin{pmatrix}
\colA{2} & \colB{2b-4i+2 \hspace{4pt}}\\
\colA{0} & \colB{2b-4 \hspace{4pt}}
\end{pmatrix}$}}
& \text{for \,} 2b-2\tau-1 < i \le b-1.
\end{cases}
\end{equation*}
Let $g_i = [G_i]$. For affine cases (the upper triangular cases) we always have $|g_i'(x)| = \frac{1}{b-2}$ for every $x \in [-1,1]$. When $i = 2b-2\tau-2$, by $\tau \leq b-2$ we have
\[
| g_i'(x) | = \left| \frac{8 \tau }{ {(-x + 2b-3 )}^2 } \right| \leq \frac{ 8 \tau }{ {(2b-4)}^2 } = \frac{ 2 \tau }{ {(b-2)}^2 } \leq \frac{ 2 (b-2) }{ {(b-2)}^2 } = \frac{2}{b-2}.
\]
Similarly, when $i = 2b-2\tau-1$, using $\frac{b-1}{2} < \tau$ and $b \geq 5$,
\[
| g_i'(x) | = \left| \frac{8 (b-\tau-1)}{ {(x + 2b-3 )}^2 } \right| \leq \frac{ 8 (b-\tau-1) }{ {(2b-4)}^2 } = \frac{ 2 (b-\tau-1) }{ {(b-2)}^2 } \leq \frac{ b - 1 }{ {(b-2)}^2 } \leq \frac{2}{b-2}.
\]

Thus, for every $i \in I$, we have $\max_{x \in [-1,1]} | g_i'(x) | \leq \frac{2}{b-2} \leq \frac{2}{3}$. Then, by the chain-rule the Lipschitz constant of $g_i \circ g_j$ for every $0 \leq i,j \leq b-1$ is bounded above by $\frac{2}{3}$. This implies that every composition of at most length $2$ of $\{g_i\}_{i \in I}$ is a strict contraction on $[-1,1]$ with Euclidean metric, and by Banach's fixed point theorem, there is a unique fixed point on $[-1,1]$ which is attracting. Hence $\mathrm{Attr}_2 \subset [-1,1]$ and $\mathrm{Rep}_2 \cap [-1,1] = \varnothing$ (by the uniqueness of fixed point on $[-1,1]$). Therefore, there can be no generalized heteroclinic connections of depth $2$.
\end{proof}

\begin{proof}[Proof of Proposition \ref{prop: thick Cantor sets are Kernel expandable}]
Suppose $(D_1, D_2)$ is non-degenerate. Then the matrices $\{A_i\}_{i \in I}$ are invertible. Let $\rho_1 = b - \#D_1$, $\rho_2 = b - \#D_2$, $m = \max\{ \rho_1, \rho_2 \} (\leq \rho(b))$, and $R = \rho_1 + \rho_2$. Fix $0 \leq i \leq b-1$ and let $A_i = \begin{pmatrix} p & q \\ r & s \end{pmatrix}$, where we recall
\begin{align*}
&p = \# \big( \left( D_1+i \right) \cap D_2 \big), \quad
q = \# \big( \left( D_1+i \right) \cap \left(D_2 + b \right) \big), \\
&r = \# \big( \left( D_1+i+1 \right) \cap D_2 \big), \quad
s = \# \big( \left( D_1 +i+1 \right) \cap \left(D_2 + b \right) \big).
\end{align*}

We first consider the combinations of $p,q,r,s$ that allow sharp estimates.
\begin{claim} \label{claim: bound on p+q etc}
We have
\begin{align*}
& b-R \leq p+q \leq b-m, \quad
b-R \leq r+s \leq b-m, \\
& b-R-1 \leq p+s \leq b-m+1, \quad
b-R-1 \leq q+r \leq b-m.
\end{align*}
\end{claim}

\begin{proof}
Let $E_1 = \{ d + i \hspace{6pt}\mathrm{mod} \hspace{4pt} b \, | \, d \in D_1\} \subset \{0, 1, \ldots, b-1\}$. Then,
\[
p+q = \# \big( \left( D_1+i \right) \cap \left( D_2 \cup \left(D_2 + b \right) \right) \big)
= \# \big( E_1 \cap D_2 \big).
\]
Since $\# E_1 = b - \rho_1$, we have $b-(\rho_1+\rho_2) \leq p+q \leq \min \{b-\rho_1, b - \rho_2\} = b - m$. The inequality for $r+s$ follows by the same proof with $i$ replaced with $i+1$.

Next, let $U_1 = \big( D_1 + i \big) \cap \{0, 1, \ldots, b-1\}$ and $V_1 = \big( D_1+i+1 - b \big) \cap \{0,1, \ldots, b-1 \}$. Then, $p = \# \big( U_1 \cap D_2 \big)$ and $s = \# \big( V_1 \cap D_2 \big)$. Since $U_1 \subset \{i, \ldots, b-1\}$ and $V_1 \subset \{0, \ldots, i\}$, we have $\# \big( U_1 \cap V_1 \big) \leq 1$. Then,
\[
p+s = \# \big( U_1 \cap D_2 \big) + \# \big( V_1 \cap D_2 \big) \leq \# D_2 + \# \big( U_1 \cap V_1 \big) \leq b - \rho_2 + 1.
\]
We define the following maps.
\begin{align*}
&\text{$u: D_1 \cap \{0, \ldots, b-1-i\} \to \mathbb{Z}$ \hspace{4pt} by \hspace{4pt} $u(x) = x+i$,} \\
&\text{$v: D_1 \cap \{ b-1-i, \ldots, b-1\} \to \mathbb{Z}$ \hspace{4pt} by \hspace{4pt} $v(y) = y + i + 1 - b$}.
\end{align*}
Then the image of $u$ and $v$ are $U_1$ and $V_1$, respectively. Thus $\#D_1 \leq \#U_1 + \#V_1 \leq \#D_1 + 1$. In particular,
\begin{align*}
p+s \leq \#U_1 + \#V_1 \leq \#D_1 + 1 = b - \rho_1 + 1.
\end{align*}

Hence $p+s \leq b-m+1$. For the lower bound, note that for any $X, Y \subset \{0,1,\ldots,b-1\}$ we have $ \# \big( X \cap Y \big) \geq \#X + \#Y - b$. Then, using $\#U_1 + \#V_1 \geq \#D_1$,
\begin{align*}
p+s
&\geq \# \big( ( U_1 \cup V_1 ) \cap D_2 \big)
\geq \# \big( U_1 \cup V_1 \big) + \#D_2 - b \\
&\hspace{50pt} = \#U_1 + \#V_1 - \#\big(U_1 \cap V_1\big) + \#D_2 - b
\geq \#D_1 -1 + \#D_2 - b
\geq b - R - 1.
\end{align*}

We are left to prove the bound for $q+r$. Let $X_1 = \big( D_1 + i - b\big) \cap \{0, 1, \ldots, b-1\}$ and $Y_1 = \big( D_1 +i+1 \big) \cap \{0,1, \ldots, b-1 \}$. Then, $q = \# \big( X_1 \cap D_2 \big)$ and $r = \# \big( Y_1 \cap D_2 \big)$. Since $X_1 \subset \{0, \ldots, i-1\}$ and $Y_1 \subset \{i+1, \ldots, b-1\}$, we have $X_1 \cap Y_1 = \varnothing$. Then,
\[
q+r = \# \big( X_1 \cap D_2 \big) + \# \big( Y_1 \cap D_2 \big) \leq \# D_2 \leq b - \rho_2.
\]
As before, define
\begin{align*}
&\text{$\widetilde{u}: D_1 \cap \{ b-i, \ldots, b-1\} \to \mathbb{Z}$ \hspace{4pt} by \hspace{4pt} $\widetilde{u}(x) = x+i-b$,} \\
&\text{$\widetilde{v}: D_1 \cap \{0, \ldots, b-2-i\} \to \mathbb{Z}$ \hspace{4pt} by \hspace{4pt} $\widetilde{v}(y) = y + i + 1$}.
\end{align*}
The image of $\widetilde{u}$ and $\widetilde{v}$ are $X_1$ and $Y_1$, respectively. Then $\#D_1 -1 \leq \# X_1 + \# Y_1 \leq \#D_1$, and
\begin{align*}
q+r \leq \#X_1 + \#Y_1 \leq \#D_1 = b - \rho_1.
\end{align*}
Therefore $q+r \leq b-m$. For the lower bound,
\begin{align*}
q+r
&\geq \# \big( ( X_1 \cup Y_1 ) \cap D_2 \big)
\geq \# \big( X_1 \cup Y_1 \big) + \#D_2 - b \\
&\hspace{50pt} = \#X_1 + \#Y_1 - \#\big(X_1 \cap Y_1\big) + \#D_2 - b
\geq \#D_1 -1 + \#D_2 - b
\geq b - R - 1.
\end{align*}
\end{proof}

Now, let 
\[
\begin{pmatrix} \alpha & \beta \\ \gamma & \delta \end{pmatrix} := F(A_i) = \frac{1}{2}
\begin{pmatrix}
p-q-r+s & p+q-r-s \\
p-q+r-s & p+q+r+s
\end{pmatrix}.
\]
By Claim \ref{claim: bound on p+q etc} we obtain
\[
\frac{m-R-1}{2} \leq \alpha \leq \frac{R-m+2}{2}, \quad \frac{m-R}{2} \leq \beta \leq \frac{R-m}{2}, \quad \delta \geq b-R.
\]
Since $R = \rho_1 + \rho_2 \leq 2m$, we obtain
\[
-\frac{m+1}{2} \leq \alpha \leq \frac{m+2}{2}, \quad -\frac{m}{2} \leq \beta \leq \frac{m}{2}, \quad \delta \geq b-2m.
\]

Let $f_i = [F(A_i)] \in \mathrm{PGL}_2(\mathbb{R})$. Fix $0 < t < 1$. Since $\delta > 0$ and $|\gamma| < \delta$ by Lemma \ref{lemma: denominator is positive}, we have for any $x \in [-t,t]$
\[
| \gamma x + \delta | \geq \delta - |\gamma| |x| \geq \delta(1-t).
\]
Then, for any $x \in [-t,t]$,
\[
|f_i(x)| = \left| \frac{ \alpha x + \beta }{ \gamma x + \delta } \right| \leq \frac{ |\alpha| |x| + |\beta| }{ \delta(1-t) } \leq \frac{ (m+2)t + m }{ 2(b-2m)(1-t) }.
\]

Now, if
\[
\frac{(m+2)t+m}{2(b-2m)(1-t)}<t,
\]
then $f_i ([-t,t]) \subset (-t,t)$ for every $i \in I$, so $[-t,t]$ is an invariant arc. Let $t = \frac{1}{2}$, then this condition is
\[
\frac{ \frac{1}{2}(m+2) + m }{ b-2m } < \frac{1}{2},
\]
which is equivalent to $5m < b - 2$, and is satisfied by $m \leq \rho(b) = \lfloor \frac{b-3}{5} \rfloor$. Thus $[-\frac12, \frac12]$ is an invariant arc. Theorem \ref{thm: main theorem 3}
implies that the family $\{A_i\}_{i=0}^{b-1}$ is simultaneously conjugate to a family of positive matrices.
Hence the Lyapunov exponent is unchanged, and Theorem \ref{thm: main theorem 1} gives a series
representation for $\lambda$. Finally, applying Theorem \ref{theorem: KP main result} completes the proof.
\end{proof}

\section{Generalizations and Limitations} \label{section: generalizations and limitations}

The algebraic/telescoping part of the Kernel-expansion argument in Theorem \ref{thm: main theorem 1} extends to invertible positive $(d+1)\times(d+1)$ matrices, but the quantitative truncation estimates used in Theorem \ref{thm: main theorem 2} are currently specific to the $2 \times 2$ case.

Consider the $d$-simplex
\[
\Delta_d^{\circ}
=
\left\{ y=(y_1,\dots,y_{d+1})\in \mathbb{R}^{d+1} \setcond y_j > 0, \; \sum_{j=1}^{d+1} y_j = 1 \right\}.
\]
Let $D=(-1,1)^d$. We define the projective ratio chart $\Phi : \Delta_d^{\circ} \to D$ by
\[
\Phi(y)_j = \frac{y_j - y_{d+1}}{y_j + y_{d+1}}
\quad (1\leq j\leq d).
\]
Suppose $\{A_i\}_{i \in I}$ is a finite family of positive $(d+1) \times (d+1)$ matrices. Denote by $F_i: \Delta_d^{\circ} \to \Delta_d^{\circ}$ the projective action of $A_i$, and define $f_i: D \to D$ by $f_i = \Phi \circ F_i \circ \Phi^{-1}$. Then, as in the $2\times 2$ case, positivity gives a uniform strict invariance: there is $0<r<1$ such that $f_i(D) \subset [-r, r]^d$ for all $i \in I$.

We equip $\Delta_d^{\circ}$ with the Hilbert projective metric, defined by
\[
d_{\mathrm{H}}(y,z)
=
\log\!\left( \frac{\max_{1\leq p\leq d+1}(y_p/z_p)}{\min_{1\leq p\leq d+1}(y_p/z_p)} \right).
\]
Consider the following operator defined on $1$-Wasserstein space associated to $d_{\mathrm{H}}$.
\[
\mathscr{H}\nu = \sum_{i\in I} w_i (F_i)_*\nu.
\]
By the Birkhoff--Hopf theorem \cite[Theorem A.4.1]{Lemmens--Nussbaum}, each $F_i$ is a strict contraction for $d_{\mathrm H}$, and hence $\mathscr{H}$ is a strict contraction in the Wasserstein distance associated to $d_{\mathrm H}$. By the completeness of the underlying space $(\Delta_d^{\circ}, d_{\mathrm{H}})$, there is a unique stationary measure $\nu^*$ on $\Delta_d^{\circ}$, and $\mathscr{H}^n\delta \to \nu^*$ for every Dirac measure $\delta$.

For $\alpha \in\mathbb{N}_0^d$ denote by $|\alpha|$ the sum of its entries. For $x \in \mathbb{R}^d$, let $x^{[\alpha]} = \prod_{j=1}^d {x_j}^{\alpha_j}$. We define $v(x\,;\,c) \in \mathbb{R}^{\mathbb{N}_0^d}$ by
\[
v_{\mathbf{0}}(x\,;\,c):=c,
\qquad
v_{\alpha}(x\,;\,c):= -\frac{(-x)^{[\alpha]}}{|\alpha|}
\quad (\alpha \ne \mathbf{0}).
\]
For each $i \in I$, define $\ell_i: D \to \mathbb{R}$ by
\[
\ell_i(x)=\log \big( \| A_i \Phi^{-1}(x) \|_1 \big).
\]
One can define the coefficients of the infinite matrix $T = (b_{\kappa,\beta})_{\kappa,\beta\in\mathbb{N}_0^d}$ so that, for every $x \in D$ and $c \in \mathbb{R}$,
\[
T v(x\,;\,c)
=
\sum_{i\in I} w_i
\big(
v(f_i(x)\, ;\, \ell_i(x)) - v(f_i(0)\, ;\, \ell_i(0))
\big).
\]
Namely, $b_{\kappa,\beta}$ are obtained by expanding the above expression component-wise in the monomial basis $\{x^{[\alpha]}\}_{\alpha \in \mathbb{N}_0^d}$. Once the above identity is available, the same telescoping argument as in the proof of Theorem \ref{thm: main theorem 1} gives a finite sum formula. That is, setting $\nu_n = \mathscr{H}^n \delta_{\Phi^{-1}(0)}$ we have
\[
\sum_{j=0}^{n-1} (T^j\boldsymbol{v})_\mathbf{0}
=
\sum_{i\in I} w_i \int_{\Delta_d^{\circ}} \log \| A_i y \|_1 \, d\nu_{n-1}(y).
\]
Therefore, using $\nu_n \to \nu^*$ and the Furstenberg--Kifer formula (analogous to equation \eqref{eq: integral form by Furstenberg Kifer}), we obtain the same series representation with $\boldsymbol{v} = \sum_{i \in I} w_i v(f_i(0)\, ;\, \ell_i(0))$.
\[
\lambda = \sum_{n=0}^{\infty} (T^n\boldsymbol{v})_\mathbf{0}.
\]

In the $2\times 2$ case, the essential ingredient in the proof of Theorem \ref{thm: main theorem 2} was Lemma \ref{lemma: real contraction implies complex contraction} together with one-variable holomorphic integral operators. In higher dimensions, this lemma is false in general: a contraction on the real cube $[-1,1]^d$ does not automatically yield a uniform contraction on the polydisk $\mathbb{D}_1^d$ (counterexamples exist). Because of this, the complex-analytic estimates used for the truncation error in the $2\times 2$ case do not presently extend.

\appendix

\section{Appendix: Proof of minor lemmas} \label{appendix: proof of minor lemmas}

\begin{proof}[Proof of Lemma \ref{lemma: Mobius classification}]
Choose a representative matrix $A = \begin{pmatrix} a & b \\ c & d \end{pmatrix} \in \mathrm{GL}_2(\mathbb{R})$ for $f$. First suppose $c=0$. Then $f(x) = \alpha x + \beta$ with $\alpha = \frac{a}{d}$, $\beta = \frac{b}{d}$. We have $\widehat{\infty}$ as its fixed point, and any fixed point on $\mathbb{R}$ must satisfy $(1- \alpha) x = \beta$.

If $\alpha = 1$ and $\beta=0$, then $f$ is the identity map, which is a contradiction.

If $\alpha = 1$ and $\beta \ne 0$, then there is no fixed point in $\mathbb{R}$, and $\widehat{\infty}$ is the unique fixed point for $f$. In this case $f(x) = x + \beta$, and we see that $f'(\widehat{\infty}) = 1$ and $f$ is parabolic.

If $\alpha \ne 1$, then $p = \frac{\beta}{1 - \alpha} \in \mathbb{R}$ is a fixed point. Direct calculation shows $f'(p) = \alpha$ and $f'(\widehat{\infty}) = \frac{1}{\alpha}$. Thus, $f$ is either an involution or hyperbolic. If it is an involution, $\alpha = -1$, and $f^2 = \mathrm{id}$.

Now suppose $c \ne 0$. Then $\widehat{\infty}$ is not a fixed point. The fixed point equation $f(x) = x$ is
\[
cx^2 + (d-a)x - b = 0.
\]
Let $\Delta = (d-a)^2 + 4bc$ be its discriminant. If $\Delta < 0$, then there are no fixed points in $\widehat{\mathbb{R}}$, and $f$ is elliptic. If $\Delta = 0$, then the fixed point of $f$ is unique and it is $q := \frac{a-d}{2c}$. Using $4bc = - (d-a)^2$, we compute
\[
f'(q) = \frac{ ad-bc }{ (cq + d)^2 } = \frac{ ad + \frac{(d-a)^2}{4} }{ \left( \frac{a-d}{2} + d \right)^2 } = 1.
\]
Thus, $f$ is parabolic.

If $\Delta > 0$, there are two distinct fixed points $p \ne q$ of $f$ in $\mathbb{R}$. Define $\phi \in \mathrm{PGL}_2(\mathbb{R})$ by
\[
\phi(x) = \frac{x-p}{x-q}
\]
and let $g = \phi \circ f \circ \phi^{-1}$. Then $g(0) = 0$ and $g(\widehat{\infty}) = \widehat{\infty}$. Thus $g(x) = \alpha x$ with some $\alpha \in \mathbb{R} \setminus \{0\}$. If $|\alpha| \ne 1$ then $g$ is hyperbolic, and so is $f$, since conjugacy preserves derivatives at fixed points. If $\alpha = 1$ then $g$ is identity and so is $f$. If $\alpha=-1$ then $g$, and thus $f$, is an involution. In this case, $g^2 = \mathrm{id}$ and $f^2 = (\phi^{-1} \circ g \circ \phi)^2 = \mathrm{id}$.
\end{proof}

\begin{proof}[Proof of Lemma \ref{lemma: bound for Kr}]
First, by Cauchy-Schwarz inequality,
\[
K(r)
\leq r \left( \frac{1}{2\pi} \int_0^{2\pi} \frac{1}{ {|1 - r e^{i\theta} |}^2} d\theta \right)^{\frac{1}{2}}
= r\left( \frac{1}{2\pi} \int_0^{2\pi} \sum_{n,m \in \mathbb{N}_0} r^{n+m}e^{i(n-m)\theta} d\theta \right)^{\frac{1}{2}}
=\frac{r}{\sqrt{1-r^2}}.
\]

Next, letting $k = \frac{2 \sqrt{r} }{ 1+r }$, we have
\[
| 1 - r e^{i\theta} | = \sqrt{ (1+r)^2 - 4r \cos^2{\frac{\theta}{2}} } = (1+r) \sqrt{ 1 - k^2 \cos^2\frac{\theta}{2} }.
\]
Letting $t = \theta/2$ and using symmetry for $t \in [0, \pi/2]$ and $t \in [\pi/2, \pi]$,
\[
K(r) = \frac{r}{\pi(1+r)} \int_0^\pi \frac{1}{ \sqrt{ 1 - k^2 \cos^2 t }} dt
= \frac{2r}{\pi(1+r)} \int_0^\frac{\pi}{2} \frac{1}{ \sqrt{ 1 - k^2 \sin^2 t }} dt.
\]

Define the complete elliptic integral of the first kind $\mathscr{K}(k)$ and the second kind $\mathscr{E}(k)$.
\[
\mathscr{K}(x) = \int_0^\frac{\pi}{2} \frac{1}{ \sqrt{ 1 - x^2 \sin^2 t }} dt, \quad \mathscr{E}(x) = \int_0^\frac{\pi}{2} \sqrt{ 1 - x^2 \sin^2 t } dt.
\]
Then we have
\[
\frac{d}{dx} \mathscr{K}(x)
= \frac{ \mathscr{E}(x) - (1-x^2)\mathscr{K}(x) }{ x(1-x^2) }
= \frac{x}{1-x^2} \int_0^{\frac{\pi}{2}} \frac{ \cos^2 t }{ \sqrt{ 1 - x^2 \sin^2 t } } dt.
\]

Now, let
\[
F(x) = \frac{\pi}{2} + \frac{1}{2} \log \frac{1}{1-x^2} - \mathscr{K}(x).
\]
Then $F(0) = 0$. Since
\[
0 \leq \int_0^{\frac{\pi}{2}} \frac{ \cos^2 t }{ \sqrt{ 1 - x^2 \sin^2 t } } dt \leq \int_0^{\frac{\pi}{2}} \cos t dt = 1,
\]
we have $F'(x) \geq 0$ for $x \in (0,1)$. Thus, $F(k) \geq 0$, $\mathscr{K}(k) \leq \frac{\pi}{2} + \frac{1}{2} \log \frac{1}{1-k^2}$, and
\[
K(r) \leq \frac{2r}{\pi(1+r)} \left( \frac{\pi}{2} + \log \frac{1+r}{1-r} \right).
\]
\end{proof}

\section{Appendix: Comparison with the transfer-operator method} \label{appendix: comparison}

Let us revisit Example \ref{example: middle fifth}. We recall that the matrices $A_0, \ldots, A_4$ are, respectively,
\[
\begin{pmatrix}
4 & 0\\
2 & 1
\end{pmatrix}
, \hspace{2pt}
\begin{pmatrix}
2 & 1\\
1 & 2
\end{pmatrix}
, \hspace{2pt}
\begin{pmatrix}
1 & 2\\
2 & 1
\end{pmatrix}
, \hspace{2pt}
\begin{pmatrix}
2 & 1\\
1 & 2
\end{pmatrix}
, \hspace{2pt}
\begin{pmatrix}
1 & 2\\
0 & 4
\end{pmatrix}.
\]
Consider the uniform weight $w_i = \frac15$. We conjugate $\{A_i\}_{i = 0}^4$ as $M_i = P A_i P^{-1}$, where $P$ is the matrix
\[
P =
\begin{pmatrix}
7 - \sqrt{105} & 7 + \sqrt{105} \\
7 + \sqrt{105} & 7 - \sqrt{105}
\end{pmatrix}.
\]
Then $\{M_i\}_{i = 0}^4$ are positive matrices. This $P$ is obtained by considering the symmetric invariant arc around $0$ that gives the smallest contraction constant for $\{f_i\}_{i=0}^4$, which can be easily determined as $\big[\,-\sqrt{7/15}, \; \sqrt{7/15}\,\big]$. We apply Kernel expansion and Transfer-operator method to the resulting $\{M_i\}_{i=0}^4$. (With this $P$, we have $r = 3/7$ in Theorem \ref{thm: precise bounds}. One can actually prove that this is the global minimum of $r$, but we omit this.)

In the Transfer-operator method, increasing the truncation level $K$ requires evaluating all $4^{K}$ products of length $K$ (not $5^K$ since $A_1 = A_3$). In contrast, Kernel expansion is essentially matrix-vector multiplications. Thus the comparison is between word enumeration entailing combinatorial explosion and polynomial-time matrix computation.

Table \ref{tab: epsilon vs N and M} shows the truncation parameters $N$ and $M$ needed to certify error at most $\vep$ using Kernel expansion. The arithmetic cost is $O(M^3)+O(NM^2)$, thus polynomial in $\log(1/\vep)$.
\begin{table}[ht]
\centering
\setlength{\tabcolsep}{9pt}
\begin{tabular}{c c c c c c c c}
\hline
Tolerated error $\vep$
& $10^{-2}$ & $10^{-4}$ & $10^{-8}$ & $10^{-16}$ & $10^{-32}$ & $10^{-64}$ & $10^{-128}$ \\
\hline
$N$ & $4$ & $10$ & $21$ & $44$ & $88$ & $174$ & $347$ \\
$M$ & $8$ & $15$ & $28$ & $51$ & $96$ & $185$ & $361$ \\
\hline
\end{tabular}
\caption{Parameters $(N,M)$ that certify
$\left| \lambda - \lambda_{N,M} \right| < \vep$, where $\lambda_{N,M}=\sum_{n=0}^{N-1}(T_M^n v^{(M)})_0$.}
\label{tab: epsilon vs N and M}
\end{table}

Table \ref{tab: Kernel expansion approximations} records the calculated values and the rigorous error bound for some $(N,M)$. On a laptop, Kernel expansion certifies the first $128$ digits of $\lambda$ in about $90$ seconds.

\begin{table}[ht]
\centering
\begin{tabular}{rll}
\hline
$(N,M)$ & $\hspace{140pt} \lambda_{N,M}$ & error bound \\
\hline
$(4,8)\;\;$
& $1.1593585196\;5887974705\;8885744885\;9366893584\;4245940103$
& $7.88 \times 10^{-3}$
\\
$(10,15)$
& $1.1593579553\;2740673129\;9210796418\;2505717953\;9205503246$
& $9.43 \times 10^{-5}$
\\
$(21,28)$
& $1.1593579553\;2718828347\;2159060603\;6912228084\;9261829560$
& $8.30 \times 10^{-9}$
\\
$(44,51)$
& $1.1593579553\;2718828347\;2158428142\;8914386399\;4864221944$
& $9.19 \times 10^{-17}$
\\
$(88,96)$
& $1.1593579553\;2718828347\;2158428142\;8914386399\;4810412424$
& $9.93 \times 10^{-33}$
\\
\hline \\[-15pt]
$(\infty, \infty)\;$
& $1.1593579553\;2718828347\;2158428142\;8914386399\;4810412424\cdots$
& \\
\hline
\end{tabular}
\caption{Approximations $\lambda_{N,M}$ and certified upper bounds on $|\lambda-\lambda_{N,M}|$. The row $(N,M)=(\infty, \infty)$ is the (certified) true value of $\lambda$.}
\label{tab: Kernel expansion approximations}
\end{table}

\begin{table}[ht]
\centering
\begin{tabular}{rll}
\hline
$K$ & $\hspace{140pt} \lambda_K$ & error bound \\
\hline
3 & $1.1595253927\;6974342060\;8204782031\;5760855407\;5634379993$ & \emph{not certified} \\
4 & $1.1593583663\;0260261032\;1044788558\;5066001879\;9571948967$ & $40.0012$ \\
5 & $1.1593579541\;5280518277\;9722298687\;0738241616\;4423268551$ & $7.035\times 10^{-1}$ \\
6 & $1.1593579553\;2756075041\;9629229785\;9895873887\;0734073801$ & $5.609\times 10^{-3}$ \\
7 & $1.1593579553\;2718898120\;2426160057\;6821245380\;8019404098$ & $1.89\times 10^{-5}$ \\
8 & $1.1593579553\;2718828303\;4636146654\;0225800301\;9330543192$ & $2.67\times 10^{-8}$ \\
9 & $1.1593579553\;2718828347\;2114507235\;7342215103\;0113949959$ & $1.60\times 10^{-11}$ \\
10 & $1.1593579553\;2718828347\;2158427231\;5080978610\;1753136751$ & $4.08\times 10^{-15}$ \\
11 & $1.1593579553\;2718828347\;2158428142\;8898432279\;0338222737$ & $4.41\times 10^{-19}$ \\
12 & $1.1593579553\;2718828347\;2158428142\;8914386876\;3239888648$ & $2.07\times 10^{-23}$ \\
\hline \\[-15pt]
$\infty$ & $1.1593579553\;2718828347\;2158428142\;8914386399\;4810412424\cdots$ & \\
\hline
\end{tabular}
\caption{Approximations $\lambda_K$ and certified upper bounds on $|\lambda-\lambda_K|$ using \cite{Jurga--Morris}, where the row $K=\infty$ is the true value of $\lambda$ calculated with our method.}
\label{tab: transfer-operator method}
\end{table}

Table \ref{tab: transfer-operator method} shows the transfer-operator approximations $\lambda_K$ and the corresponding Jurga–Morris \cite{Jurga--Morris} error bounds, where the row $K = \infty$ records the (certified) true value of $\lambda$. They prove that $|\lambda - \lambda_K| = O( e^{-\gamma K^2} )$ with some $\gamma>0$. However, computing $\lambda_K$ requires evaluating $M_{i_K}\cdots M_{i_1}$ for all words of length $K$, so the cost is subexponential in $\log(1/\vep)$, namely $O ( \exp{ ( \alpha \sqrt{ \log{ \left( 1/ \vep \right) } } ) } )$ for some $\alpha > 0$. To obtain rigorously certified error bound of $10^{-128}$, the needed $K$ is $27$. In this comparison, we did not use the sharper bound in \cite{Pollicott20}, where the matrices are assumed to be in $\mathrm{SL}_2(\mathbb{R})$.

\section*{Acknowledgements}
I would like to express my sincere gratitude to my supervisor, Masaki Tsukamoto, for his guidance and support. I am also grateful to Shintaro Suzuki for suggesting the application to random Fibonacci sequences.

ChatGPT (OpenAI) was used to review the article and polish it, and carry out numerical experiments. It was not used to generate statements or proofs.

I further wish to thank the wonderful people around me for their support, as well as the academic community.

\vspace{0.5cm}

\address{
Department of Mathematics, Kyoto University, Kyoto 606-8501, Japan}

\textit{E-mail}: \texttt{alibabaei.nima.28c@st.kyoto-u.ac.jp}

This work was supported by JSPS KAKENHI Grant Number 25KJ1473.

\end{document}